\documentclass[11pt,leqno]{amsart}

\makeatletter
\@namedef{subjclassname@2020}{\textup{2020} Mathematics Subject Classification}
\makeatother

\usepackage[top=2.5 cm,bottom=2 cm,left=2.5 cm,right=2 cm]{geometry}
\usepackage[utf8]{inputenc}
  \usepackage{color}

\usepackage{esint}
\usepackage{graphicx}
\usepackage{hyperref}
\usepackage{color}
\usepackage{epsfig}
\usepackage{ulem}
\usepackage{lettrine}
\usepackage{verbatim}
\usepackage{array}
\usepackage{varwidth}
\usepackage{mathptmx} 
\usepackage{setspace}
\usepackage{lineno}

\usepackage[numbers,compress]{natbib}

\numberwithin{equation}{section}

\newtheorem{example}{Example}[section]
\newtheorem{definition}[example]{Definition}
\newtheorem{theorem}[example]{Theorem}
 \newtheorem{proposition}[example]{Proposition}
\newtheorem{lemma}[example]{Lemma}

\newtheorem*{maintheorem*}{Main Theorem}
\allowdisplaybreaks
\numberwithin{equation}{section}

\renewcommand{\i}{\ifmmode\mathit{\mathchar"7010 }\else\char"10 \fi}
\renewcommand{\j}{\ifmmode\mathit{\mathchar"7011 }\else\char"11 \fi}
\newcommand{\R}{\mathbb{R}}
\newcommand{\N}{\mathbb{N}}

\newcommand{\norm}[1]{\left\|#1\right\|}

\newcommand{\weak}{\rightharpoonup}

\newcommand{\pt}{\partial_t}

\newcommand{\ptt}{\partial_{tt}^2}

\newcommand{\vfi}{\varphi}

\newcommand{\M}{{\mathcal M}}

\def\begi{\begin{itemize}}
\def\endi{\end{itemize}}
\def\bega{\begin{array}}
\def\enda{\end{array}}

\newcommand{\bv}{{\bf v}}

\def\u{\mathbf{u}}
\def\wu{{\widetilde{\mathbf{u}}}}

\def\vv{\mathbf{v}}

\def\b{\mathbf{b}}
\def\x{\mathbf{x}}
\def\f{\mathbf{f}}

\newenvironment{Assumptions}
{%

\begin{enumerate}}%
{\end{enumerate}}

{%

\begin{enumerate}}%
{\end{enumerate}}

\begin{document}

\title[A P-(EC)$^k$ scheme for nonlinear Peridynamics]{A numerical framework for nonlinear Peridynamics on two-dimensional manifolds based on implicit P-(EC)$^k$ schemes}

\author[A. Coclite]{Alessandro Coclite}
\author[G. M. Coclite]{Giuseppe Maria Coclite}
\author[F. Maddalena]{Francesco Maddalena}
\author[T. Politi]{Tiziano Politi}

\address[A. Coclite, T. Politi]{Dipartimento di Ingegneria Elettrica e dell'Informazione (DEI), \newline Politecnico di Bari, Via Re David 200 -- 70125 Bari, Italy}
\email{alessandro.coclite@poliba.it, tiziano.politi@poliba.it}

\address[G. M. Coclite, F. Maddalena]{ Dipartimento di Meccanica, Matematica e Management (DMMM), \newline Politecnico di Bari, Via Re David 200 -- 70125 Bari, Italy}
\email{giuseppemaria.coclite@poliba.it, francesco.maddalena@poliba.it}

\date{\today}



\begin{abstract} 

In this manuscript, an original numerical procedure for the nonlinear peridynamics on arbitrarily--shaped two-dimensional (2D) closed manifolds is proposed. When dealing with non parameterized 2D manifolds at the discrete scale, the problem of computing geodesic distances between two non-adjacent points arise. Here, a routing procedure is implemented for computing geodesic distances by re-interpreting the triangular computational mesh as a non-oriented graph; thus returning a suitable and general method. Moreover, the time integration of the peridynamics equation is demanded to a P-(EC)$^k$ formulation of the implicit $\beta$-Newmark scheme. The convergence of the overall proposed procedure is questioned and rigorously proved. Its abilities and limitations are analyzed by simulating the evolution of a two-dimensional sphere. The performed numerical investigations are mainly motivated by the issues related to the insurgence of singularities in the evolution problem. The obtained results return an interesting picture of the role played by the nonlocal character of the integrodifferential equation  in the intricate processes leading to the spontaneous formation of singularities in real materials.
\end{abstract}

\maketitle

\section{\bf Introduction}
\label{sec:intro}

Peridynamics was initiated by S.A. Silling in~\cite{Sill} by introducing a genuinely nonlocal approach to continuum mechanics based on long-range internal forces in place of the classical contact forces ruled by the Cauchy stress (see also \cite{MR3831320, MR2804637, MR3288460, MR3061151, MR3297136, EP, MR2375050,  E, MR0307573, KRONER1967731, LP, LP1, SILLING201073, MR2592410, MR2348150, MR2430855}). The main motivations of this new formulation of mechanical problems originate from the desire to find analytical descriptions of phenomena ascribable to the so-called non-smooth evolution processes affecting the behavior of real-world materials. Indeed, the spontaneous creation of singularities like cracks, damages, or defects and their evolution, as well as the dispersive character of the waves propagation, are the main sources of trouble to establish suitable mathematical models which frame these problems in a coherent and unified framework. These phenomena suggest the presence of many scales involved in the underlying physics and this represents a very challenging problem in the analysis. {From the numerical point of view, peridynamics models can be grossly divided into two macro-areas: finite element models~\cite{lipton14,lipton16,jha18,jha19,jha21,lipton21} and mesh-free methods~\cite{silling2005meshfree,gu2017voronoi,bessa2014meshfree} or, equivalently, quadrature methods \cite{javili2019peridynamics,shojaei2022hybrid}. Indeed, such a division is solely for summarizing purposes, and a plethora of subdivisions, parallelisms, and couplings can be found in the literature. Moreover, further numerical methods, like spectral methods \cite{LP1,CFLMP,jafarzadeh2020efficient,LP}, Boundary Element Methods \cite{liang2021boundary}, have been developed recently in the context of peridynamics, enlarging the range of available numerical tools.}

The peridynamic formulation leads to the study of nonlinear integrodifferential equations which encapsulate the main ingredients to focus these problems~\cite{dimola2022}. In this paper, the authors pursue previous studies 
(see \cite{CDFMRV, CDFMV, Coclite_2018, CFLMP}) in which the governing equation of motion is ruled by a singular integral in which the nonlocal effects are weighted through a suitable exponent while the nonlinearity in the constitutive assumptions plays a crucial role in the evolution problem. More precisely, here the authors investigate in detail the evolution of a material body whose resting configuration is a two-dimensional closed manifold, to capture the insurgence of singularities deriving from possible energy localizations. This program is made possible by the functional setting in which the well-posedness of the Cauchy problem is proved. 
\textcolor{blue}{ Note that peridynamic kernels should focus on three grandstanding ingredients~\cite{dimola2022}: {\it nonlocality},  {\it impenetrability of the medium},   {\it nonlinearity and spontaneous fracture generation}. For granting the latter, such kernel needs to be singular (extending so the Silling's model based on nonsingular kernels) and thanks to its singular nature the energy space is $W^{\alpha,p}$ ($\alpha$ being the nonlocality parameter and $p$ the nonlinearity exponent); so that tailoring solutions exposing discontinuities by the measure of $N-1$ (with $N$ being the dimension of the space). }

The investigations are carried out through a geometric discretization of the original problem and implementing a numerical scheme that addresses the full nonlinear problem without any approximation and can deal with any smooth closed two-dimensional manifold. This last feature constitutes, as far as the authors knowledge, an interesting novelty. {In this work only closed manifolds are considered. 
\textcolor{blue}{As well known, a rigorous procedure for including boundary conditions in peridynamics is still an open problem. 
Nevertheless there are many recent rigorous studies on nonlocal trace theories \cite{zbMATH07505259,zbMATH07527316} which constitute a suitable framework to cast the present problem.
Some formal strategies were developed in \cite{prudhomme2020,zhao2020algorithm,chen2021peridynamics,behera2022imposition}}. 
Future studies will include the description of the boundaries and the role of the choice of the specific strategy for enforcing such boundary conditions.} The paper is organized as follows. In Section~\ref{evolpb} we prove the well-posedness of the Cauchy problem for the peridynamic evolution and introduce a suitable discretization of the problem. 
In Section~\ref{sec:estimates} we prove the convergence of the solutions of the discretized problem to the ones of the continuous problem. Section~\ref{numscheme} is devoted to the study of the numerical scheme based on the approximation of the geodesic distance on the manifold providing a pseudo-code for such computation. The full discretization of the semi-discrete Cauchy problem is obtained by using an implicit P-(EC)$^k$ $\beta-$Newmark II order scheme. In Section~\ref{exp} we perform detailed numerical experiments focusing on the analysis of the evolution of the material body, in dependence on the two key parameters $p$ and $\alpha$. In particular,  the case of assigned initial velocity without external forces and the case of assigned uniaxial load are critically studied. By tracking in time the surface stretching, and the relative internal energy, the authors deepen the knowledge of the model peculiarities.   

\section{\bf The evolution problem}
\label{evolpb}
In \cite{Coclite_2018} the Cauchy problem related to a very general model of nonlocal continuum mechanics, inspired by the seminal work by  Silling~\cite{Sill},  was studied and the analytical aspects concerning global solutions in energy space were exploited in the framework of nonlinear hyperelastic constitutive assumptions for an unbounded domain. In this paper, we consider the peridynamic evolution of a material body whose initial configuration is given by a  two-dimensional smooth closed manifold  $\M\subset \R^3$. More precisely,  the motion of the body with constant density $\rho>0$ is modeled (see \cite{CDMV})
by the initial-value problem 
\begin{equation}
\label{eq:CP}
\begin{cases}
\rho\partial_{tt} \u(\x,t)=(K\u(\cdot,t))(\x)+\b(\x,t),&\quad \x\in \M,\: t >0,\\
\u(\x,0)=\u_0(\x),\:\partial_t \u(\x,0)=\mathbf{v}_0(\x),&\quad \x\in\M,
\end{cases}
\end{equation}
where
\begin{equation}
\label{eq:K}
(K\u(\cdot,t))(\x):= -\int_{\M\cap B_\delta(\x)}\mathbf{f}(\x'-\x,\,\u(\x',t)-\u(\x,t))\,d\mathcal{H}^2(\x'),\quad {\mbox{ for every }}\,\x\in\M,
\end{equation}
for a given $\delta>0$, where $\mathcal{H}^2$ is the Hausdorff two-dimensional measure and $B_\delta(\x)\subset \R^3$ is the open ball centered at $\x$ by radius $\delta$. For every $t\ge0$ the unknown $\R^3$-valued  function $\u:\M\times \R_+\rightarrow\R^3$ represents  the displacement vector field. From the point of view of peridynamics, the parameter~$\delta$ takes into account the finite horizon of the nonlocal bond which is governed by the long-range interaction integral~$K$. 

The pairwise force interaction function $\mathbf{f}: (\mathbb{R}^{3}\setminus\{\mathbf 0\})\times\mathbb{R}^3\to \R^3$
is supposed to satisfy the following general constitutive assumptions:

\begin{Assumptions}
\item  $\f\in C^1((\mathbb{R}^{3}\setminus\{\mathbf 0\})\times\mathbb{R}^3;\mathbb{R}^3)$;
\item \label{ass:f.2} {$\f(\x-\x',\u(\x)-\u(x'))=-\f(\x'-\x,\u(\x',t)-\u(x,t)),$ for every $(\x'-\x,\,\u(\cdot,t))\in(\mathbb{R}^{3}\setminus\{\mathbf 0\})\times\mathbb{R}^3$;}
\item \label{ass:f.5} there exists a scalar function $\Phi\in C^2((\mathbb{R}^{3}\setminus\{\mathbf 0\})\times\mathbb{R}^3)$ such that $\f=\nabla_{\u}\Phi$.
\end{Assumptions}

With this respect, assumption~\ref{ass:f.2} can be seen as a counterpart of Newton's Third Laws of Motion (the Action-Reaction Law). Also, assumption~\ref{ass:f.5} states that the material is hyperelastic. Moreover, thanks to \ref{ass:f.5},  the  energy 
\begin{equation}
\label{eq:Egen}
E(t)=\frac{\rho}{2}\int_{\M} |\pt \u(\x,t)|^2d\mathcal{H}^2(\x)+
\frac{1}{2} \int_{\M}\int_{\M\cap B_\delta(\x)}
\Phi(\x-\x', \u(\x,t)-\u(\x',t))\,d\mathcal{H}^2(\x')\,d\mathcal{H}^2(\x)
\end{equation}
is preserved during the motion, i.e. if $\u$ is a smooth solution of \eqref{eq:CP} then $E'(t)=0$.

In this paper we focus on the following choice for $\Phi$,
\begin{equation}
\label{eq:phigen}
{\Phi(\x'-\x,\u(\x')-\u(x))=\frac{\kappa}{p} \frac{|\u(\x')-\u(\x)|^p}{d_{\M}(\x',\x)^{2+\alpha p}}\, , \quad {\mbox{ for every }}\, (\x'-\x,\,\u(\cdot))\in(\mathbb{R}^{3}\setminus\{\mathbf 0\})\times\mathbb{R}^3}\, ,
\end{equation}
where $\kappa,\, p,\,\alpha$ are constants such that
\begin{equation}
\label{eq:constants}
\kappa>0,\qquad 0<\alpha<1,\qquad   p\ge 2,
\end{equation}
and $d_{\M}$ is the geodesic distance on the manifold $\M$.

In light of these assumptions we can write explicitly the operator $K$ as follows
\begin{equation}
\label{eq:operator}
(K\u(\cdot,t))(\x):= {\kappa}\int_{\M\cap B_\delta(\x)}
\frac{|\u(\x',t)-\u(\x,t)|^{p-2}}{d_{\M}(\x',\x)^{2+\alpha p}}(\u(\x',t)-\u(\x,t))\,d\mathcal{H}^2(\x'),\quad {\mbox{ for every }}\,\x\in\M,
\end{equation}
and the energy associated to \eqref{eq:CP} reads:
\begin{equation}
\label{eq:energy-alpha}
E(t)=\frac{\rho}{2}\int_{\M} |\pt \u(\x,t)|^2d\mathcal{H}^2(\x)+
\frac{\kappa}{2p} \int_{\M}\int_{\M\cap B_\delta(\x)}
\frac{|\u(\x,t)-\u(\x',t)|^p}{d_{\M}(\x,\x')^{2+\alpha p}}\,d\mathcal{H}^2(\x')\,d\mathcal{H}^2(\x).
\end{equation}

\subsection{Functional Spaces and Definition of Dissipative Solution}
\label{subsec:spaces}
~\newline
The natural functional  space in which the initial value problem \eqref{eq:CP}  is well posed is the one with finite energy, so 
in this paper we use the space $\mathcal{W}$  defined trough the norm
\begin{equation}
\label{eq:norm}
\norm{\u}_{\mathcal{W}}=\norm{\u}_{L^2}+
\left(\int_{\M}\int_{\M\cap B_\delta(\x)}
\frac{|\u(\x,t)-\u(\x',t)|^p}{d_{\M}(\x,\x')^{2+\alpha p}}\,d\mathcal{H}^2(\x')\,d\mathcal{H}^2(\x)\right)^{1/p}
\end{equation}
which is a closely related to the Sobolev fractional space $W^{\alpha,p}$ \cite{DNPV}.
Arguing as in {\cite[Lemmas 2.1 and 2.2]{Coclite_2018}}, the following result can be proved: 

\begin{lemma}
\label{lm:compembed}
Assume \eqref{eq:constants} and \eqref{eq:operator}. The following statements hold
\begin{itemize}
\item[$i$)] $\mathcal{W}\hookrightarrow\hookrightarrow L^2(\M;\R^3).$
\item[$ii$)] For every $\u,\,\bv\in\mathcal{W}$ we have that $(K\u)\bv\in L^1(\M).$
\item[$iii$)]  For every sequence $\{\u_h\}_h\subset \mathcal{W}$ and $\u\in\mathcal{W}$, if $
\u_h\weak \u$ weakly in $\mathcal{W}$, then $(K\u_h) \to (K\u)$ in the sense of distributions on $\M$.
\end{itemize}
\end{lemma}

On the initial data and the source we assume
\begin{equation}
\label{eq:assinit}
\u_0\in\mathcal{W},
\qquad \vv_0\in L^2(\M;\R^3),\qquad \b\in L^1_{loc}(0,{T};L^2(\M;\R^3))\, , T>0\, .
\end{equation}

Finally, in order to improve the readability we define the following space
\begin{equation*}
\mathcal{X}=\left\{\u:\M\times [0,\infty) \to\mathbb{R}^3;\>\>
\begin{aligned}& \u \in L^\infty(0, T ; \mathcal{W} ),\, T > 0 \\
&\pt\u\in L^\infty(0, T ; L^2(\M;\R^3)),\, T > 0\end{aligned}\right\}
\end{equation*}

Let us introduce the following notion of weak solution for \eqref{eq:CP}.

\begin{definition}
\label{def:sol-alpha}
A function $\u:\M\times[0,\infty)\to\R^3$ is a dissipative solution of \eqref{eq:CP} if
\begin{equation}
\label{eq:sol-weak-space}
\u\in\mathcal{X}
\end{equation}
{and, for every test function $\vfi\in C^\infty(\M\times\R;\R^3)$ with compact support, the following hold true}
\begin{equation}
\label{eq:sol-weak-alpha}
\begin{split}
\int_0^\infty\!\!\!\!\int_{\M}& (\rho\u(\x,t)\ptt\vfi(\x,t)-(K \u(\cdot,t))(\x)\vfi (\x,t) )\,dt\,d\mathcal{H}^2(\x)\\
&-\int_{\M} \rho\vv_0(\x)\vfi(\x,0)d\mathcal{H}^2(\x)+\int_{\M} \rho\u_0(\x)\pt\vfi(\x,0)d\mathcal{H}^2(\x)\\
=&\int_0^\infty\!\!\!\!\int_{\M} \b(\x,t)\vfi(\x,t) \,dt\,d\mathcal{H}^2(\x)
\end{split}
\end{equation}
and 
\begin{equation}
\label{eq:sol-diss-alpha}
E(t)\le \left(\sqrt{E(0)}+\displaystyle\frac{1}{2}\int_0^t\norm{\b(\cdot,s)}_{L^2}ds\right)^2,
\end{equation}
for every $t\ge0$, where
\begin{equation*}
 E(0)=\frac{\rho}{2}\int_{\M} |\vv_0|^2d\mathcal{H}^2(\x)+
\frac{\kappa}{2p} \int_{\M}\int_{\M\cap B_\delta(\x)}
\frac{|\u_0(\x)-\u_0(\x')|^p}{d_{\M}(\x,\x')^{2+\alpha p}}\,d\mathcal{H}^2(\x')\,d\mathcal{H}^2(\x).
\end{equation*}
\end{definition}

{In the next (sub)section, we will prove the existence of dissipative solutions of peridynamics equation of motion and we will show that the proposed semi-discrete approximation converges to such solution.}

\subsection{Discretization and convergence results}
\label{subsec:num+main}

In this section, we are going to exploit a semi-discrete numerical scheme for \eqref{eq:CP} on a triangular mesh whose nodes lays on $\M$. Such mesh is obtained by \textcolor{red}{evenly} distributing a finite number of vertices on the manifold, thus returning a triangular discretization with characteristic size $\Delta x$. \textcolor{red}{As widely known, such discretization is obtainable (for an a-priori-fixed number of vertices) onto manifolds diffeomorphic to a sphere. While, in general, a tessellation based on as-equilateral-as-possible triangles can be obtained for connected manifolds by minimally varying the desired number of vertices or the triangles' characteristic size $\Delta x$. The latter corresponds to a reference length since there is no requirement for the similarity between different triangles. Moreover,} due to the nonlocal nature of \eqref{eq:CP}, $\Delta x$ must be selected coherently with the size of the horizon $\delta$. Note that, for $p > 2$, in the limit $\delta, \Delta x \to 0$ the proposed scheme should reasonably guarantee the asymptotic compatibility to the local nonlinear elasticity description, in the same spirit of~\cite{tian2014}. However, the detailed investigation of this limit will be object of future works.

{Let $\mathcal{N}\subset\M$ be a countable dense subset and for every $n\in\N\setminus\{0\}$ let 
$$\mathcal{N}_n\subset \mathcal{N}\quad\text{be such that}\quad\#(\mathcal{N}_n)=n+1,\>\>\>\mathcal{N}_n\subset\mathcal{N}_{n+1},\>\>\>\bigcup_{n}\mathcal{N}_n=\mathcal{N}.$$
We label $\mathcal{N}_n$ as follows
$$\mathcal{N}_n=\big\{\x_{n,i}\in\mathcal{N},\, i\in\{0,...,n\}\big\}.$$
Clearly, $$i\neq \ell\Rightarrow \x_{n,i}\neq\x_{n,\ell}.$$
\textcolor{red}{Following \cite{MR3620141}, we assume that there exist an admissible triangulation
$T_0,..., T_{k_n}\subset \R^3$, whose vertices belong to $\mathcal{N}_n$ and saturate it, such that
their interiors are pairwise disjoint. For every $i\in\{0,..., k_n\}$, let $h_{T_i}$ be the diameter of $T_i$ and $\rho_{T_i}$ be the  diameter of the largest ball contained in $T_i$.
We require that the triangulation is regular and localquasi-uniform, namely
\begin{align}
\label{eq:reg}
&\exists\sigma>0\>\text{s.t.}\>\forall n\in\N\setminus\{0\}\>\forall i\in\{0,..., k_n\}:h_{T_i}\le \sigma \rho_{T_i},\\
\label{eq:locqu}
&\exists\lambda>0\>\text{s.t.}\>\forall n\in\N\setminus\{0\}\>\forall i,j\in\{0,..., k_n\},\,i\not=j:h_{T_i}\le \lambda h_{T_j}.
\end{align}}

We set 
$$\M_n=\bigcup_{j=0}^{k_n}T_j\, ,$$
being, indeed, $\M_n$ the discrete approximation of $\M$ composed by the $k_n$ triangles generated through $\mathcal{N}_n$.}

In order to discretize \eqref{eq:CP} let us define the approximations $\u_{0,n}$, $\vv_{0,n}$ and $\b_n$   of $\u_0$, $\vv_0$ and $\b$, such that
\begin{equation}
\label{eq:assinitial_n} 
\begin{split}
&\{\u_{0,n}\}_n\subset C(\M_n;\R^3)\cap \mathcal{W},\quad \{\vv_{0,n}\}_n\subset C(\M_n;\R^3),\quad \{\b_{n}\}_n\subset C([0,\infty)\times\M_n;\R^3),\\
&\text{$\u_{0,n}\to\u_{0}$ a.e. and in $\mathcal{W}$ as $n\to\infty$},\\
&\text{$\vv_{0,n}\to\vv_{0}$ a.e. and in $L^2(\M;\R^3)$ as $n\to\infty$},\\
&\text{$\b_{n}\to\b$ a.e. and in $L^1_{loc}(0,\infty;L^2(\M;\R^3))$ as $n\to\infty$}.
\end{split}
\end{equation}

We look for the time dependent $\R^3$ valued functions
\begin{equation}
\label{eq:unij}
t \ge 0 \mapsto \{\u_{n,i}(t)\}_{i=0,...,n}
\end{equation}
solving the following ODEs system
\begin{equation}
\label{eq:CPnum}
\begin{cases}
\ddot \u_{n,i}(t)=(K_{n}\u_n(\cdot,t))(\x_{n,i})+\b_n(\x_{n,i},t),&\quad  t >0,\\
\u_{n,i}(0)=\u_{0,n}(\x_{n,i}),\:\dot \u_{n,i}(0)=\mathbf{v}_{0,n}(\x_{n,i}),&{}
\end{cases}
\end{equation}
where 
$\u_n:{\M_n}\times[0,\infty)\to\R^3$ is the unique continuous piecewise affine function  on ${\M_n}$ such that 
\begin{equation}
\label{eq:franco}
\begin{split}
&\u_n(\x_{n,i},t)=\u_{n,i}(t),\qquad \hbox{\rm for every } {i=0,...,n},\\
&\text{{$\u_n(\cdot,t)$ affine on all triangles $T_0,...,T_{k_n}$}}
\end{split}
\end{equation}
and 
\begin{align}
\label{eq:operatornum}
&(K_{n}\u_n(\cdot,t))(\x_{n,i}):= {\kappa}\,{\sum_{j\in \tau_{n,i}}\sum_{\x_{n,\ell}\in T_j\cap \mathcal{N}_n}\frac{\mathcal{H}^2(T_j)}{3}}
\frac{|\u_{n,\ell}(t)-\u_{n,i}(t)|^{p-2}}{d_{\M}(\x_{n,\ell},\x_{n,i})^{2+\alpha p}}(\u_{n,\ell}(t)-\u_{n,i}(t)),\\
\label{eq:ball}
& {{B}_{n,\delta}(\x_{n,i}):=\{\x_{n,\ell}:d_{\M}(\x_{n,\ell},\x_{n,i})<\delta\}},\\
\label{eq:tringles}
&{\tau_{n,i}:=\Big\{j\in\{0,...,k_n\}: T_j\cap\mathcal{N}_n\subset  {B}_{n,\delta}(\x_{n,i})\Big\}}.
\end{align}

The existence and uniqueness of the solution  $\u_n$ to \eqref{eq:CPnum} follows by the  Cauchy-Lipchitz Theorem.

We have the following existence and convergence results.

\begin{theorem}
\label{th:main} Assume \eqref{eq:constants} and \eqref{eq:assinit}.
There exist a function $\u\in\mathcal{X}$ and a subsequence $\{\u_{n_h}\}_h$ of $\{\u_n\}_n$ defined in  \eqref{eq:franco} such that
\begin{itemize}
\item[{\bf (A)}]
$\u_{n_h}\to \u$ a.e. and in $L^q(0,T;L^2(\M;\R^3)),\,T>0,\,1\le q<\infty,$ as $h\to\infty$,
\item[{\bf (B)}]
$\u$ is a dissipative solution of \eqref{eq:CP} in the sense of Definition \ref{def:sol-alpha}.
\end{itemize}
Moreover, if $p=2$, we have also that
\begin{itemize}
\item[{\bf (C)}] $\u$ is the unique dissipative solution of \eqref{eq:CP} in the sense of Definition \ref{def:sol-alpha},
\item[{\bf (D)}] if $\u$ and $\wu$
 are the unique dissipative solutions of \eqref{eq:CP} in the sense of Definition \ref{def:sol-alpha}
obtained in correspondence of the initial data  $(\u_0,\,\bv_0)$ and $(\wu_0,\,\widetilde{\bv}_0)$,
the following estimate holds
\begin{equation}
\label{eq:stab}
\begin{split}
\frac{\rho}{2}\int_{\M}& |\pt \u(\x,t)-\pt\wu(\x,t) |^2d\mathcal{H}^2(\x)+\\
&+\frac{\kappa}{2p} \int_{\M}\int_{\M\cap B_\delta(\x)}
\frac{|(\u(\x,t)-\wu(\x,t))-(\u(\x',t)-\wu(\x',t))|^p}{d_{\M}(\x,\x')^{2+\alpha p}}\,d\mathcal{H}^2(\x')\,d\mathcal{H}^2(\x)\\
\le&\frac{\rho}{2}\int_{\M} |\pt \bv_0(\x)-\pt\widetilde{\bv}_0(\x) |^2d\mathcal{H}^2(\x)+\\
&+\frac{\kappa}{2p} \int_{\M}\int_{\M\cap B_\delta(\x)}
\frac{|(\u_0(\x)-\wu_0(\x))-(\u_0(\x')-\wu_0(\x'))|^p}{d_{\M}(\x,\x')^{2+\alpha p}}\,d\mathcal{H}^2(\x')\,d\mathcal{H}^2(\x)\\
\end{split}
\end{equation}
for every $t\ge0$.
\end{itemize}
\end{theorem}

\section{\bf Proof of Theorem \ref{th:main}}
\label{sec:estimates}

The proof of Theorem \ref{th:main} follows after the following  preliminary results.

\begin{proposition}
\label{lm:exists}
Assume \eqref{eq:constants} and \eqref{eq:assinit}. There exist a function $\u\in\mathcal{X}$ and a subsequence $\{\u_{n_h}\}_h$ of $\{\u_n\}_n$ defined in  \eqref{eq:franco} such that
\begin{align}
\label{eq:exist1}
&\text{$\u_{n_h}\to \u$ a.e. and in $L^q(0,T;L^2(\M;\R^3)),\,T>0,\,1\le q<\infty,$ as $h\to\infty$,}\\
\label{eq:exist2} &\text{$\u$ is a dissipative solution of \eqref{eq:CP} in the sense of Definition \ref{def:sol-alpha}.}
\end{align}
\end{proposition}
The proof of the previous proposition will be given after some lemmas.

We will denote with $C$ and $c$ all the constants independent {of $\delta$}.

\begin{lemma}[{\bf Discrete energy estimate}]
\label{lm:enest}
Let $\{\u_{n,i}\}$ be the solution of \eqref{eq:CPnum}. The  estimate 
\begin{equation}
\label{eq:enest}
 E_n(t)\le  \left(\sqrt{E_n(0)}+\frac{1}{\sqrt{2\rho}}\int_0^t\left({\sum_{j=0}^{k_n}\sum_{\x_{i,n}\in T_j\cap\mathcal{N}_n}\frac{\mathcal{H}^2(T_j)}{3}|\b_n(\x_{n,i},t)|^2}\right)^{1/2}dt\right)^2
\end{equation}
holds true for every $t>0$, where
\begin{align*}
E_n(t)=&{\rho\sum_{j=0}^{k_n}\sum_{\x_{i,n}\in T_j\cap\mathcal{N}_n}\frac{\mathcal{H}^2(T_j)}{3}\frac{|\dot \u_{n,i}(t)|^2}{2}+
\frac{\kappa }{p}\sum_{j=0}^{k_n}\sum_{\x_{i,n}\in T_j\cap\mathcal{N}_n}
\sum_{k\in \tau_{n,i}}\sum_{\x_{n,\ell}\in T_k\cap \mathcal{N}_n}
\frac{\mathcal{H}^2(T_j)\mathcal{H}^2(T_k)}{9}
\frac{|\u_{n,\ell}(t)-\u_{n,i}(t)|^{p}}{d_{\M}(\x_{n,\ell},\x_{n,i})^{2+\alpha p}}},\\
E_n(0)=&{\rho\sum_{j=0}^{k_n}\sum_{\x_{i,n}\in T_j\cap\mathcal{N}_n}\frac{\mathcal{H}^2(T_j)}{3}\frac{|\vv_{0,n,i}|^2}{2}+
\frac{\kappa }{p}\sum_{j=0}^{k_n}\sum_{\x_{i,n}\in T_j\cap\mathcal{N}_n}
\sum_{k\in \tau_{n,i}}\sum_{\x_{n,\ell}\in T_k\cap \mathcal{N}_n}
\frac{\mathcal{H}^2(T_j)\mathcal{H}^2(T_k)}{9}
\frac{|\u_{0,n,\ell}-\u_{0,n,i}|^{p}}{d_{\M}(\x_{n,\ell},\x_{n,i})^{2+\alpha p}}.}
\end{align*}
\end{lemma}

\begin{proof}
Multiplying \eqref{eq:CPnum} by $\dot \u_{n,i}$ and summing over $i$ we get
{\begin{equation}
\label{eq:en0}
\begin{split}
\frac{d}{dt}\rho \sum_{j=0}^{k_n}\sum_{\x_{i,n}\in T_j\cap\mathcal{N}_n}\frac{\mathcal{H}^2(T_j)}{3}\frac{|\dot \u_{n,i}(t)|^2}{2}
=&\underbrace{\sum_{j=0}^{k_n}\sum_{\x_{i,n}\in T_j\cap\mathcal{N}_n}\frac{\mathcal{H}^2(T_j)}{3}(K_{n}\u_n(\cdot,t))(\x_{n,i})\cdot \dot \u_{n,i}(t)}_{I_1}\\
&+\underbrace{\sum_{j=0}^{k_n}\sum_{\x_{i,n}\in T_j\cap\mathcal{N}_n}\frac{\mathcal{H}^2(T_j)}{3}\b_n(\x_{n,i},t)\cdot \dot \u_{n,i}(t)}_{I_2}.
\end{split}
\end{equation}}
Using the definition of the operator $K_n$ we get
{\begin{align*}
I_1=&\kappa \sum_{j=0}^{k_n}\sum_{\x_{i,n}\in T_j\cap\mathcal{N}_n}
\sum_{k\in \tau_{n,i}}\sum_{\x_{n,\ell}\in T_k\cap \mathcal{N}_n}
\frac{\mathcal{H}^2(T_j)\mathcal{H}^2(T_k)}{9}
\frac{|\u_{n,\ell}(t)-\u_{n,i}(t)|^{p-2}}{d_{\M}(\x_{n,\ell},\x_{n,i})^{2+\alpha p}}(\u_{n,\ell}(t)-\u_{n,i}(t))\cdot \dot \u_{n,i}(t)\\
=&\frac{ \kappa}{2}\sum_{j=0}^{k_n}\sum_{\x_{i,n}\in T_j\cap\mathcal{N}_n}
\sum_{k\in \tau_{n,i}}\sum_{\x_{n,\ell}\in T_k\cap \mathcal{N}_n}
\frac{\mathcal{H}^2(T_j)\mathcal{H}^2(T_k)}{9}
\frac{|\u_{n,\ell}(t)-\u_{n,i}(t)|^{p-2}}{d_{\M}(\x_{n,\ell},\x_{n,i})^{2+\alpha p}}(\u_{n,\ell}(t)-\u_{n,i}(t))\cdot \dot \u_{n,i}(t)\\
&+\frac{\kappa }{2}\sum_{j=0}^{k_n}\sum_{\x_{\ell,n}\in T_j\cap\mathcal{N}_n}
\sum_{k\in \tau_{n,\ell}}\sum_{\x_{n,i}\in T_k\cap \mathcal{N}_n}
\frac{\mathcal{H}^2(T_j)\mathcal{H}^2(T_k)}{9}
\frac{|\u_{n,\ell}(t)-\u_{n,i}(t)|^{p-2}}{d_{\M}(\x_{n,\ell},\x_{n,i})^{2+\alpha p}}(\u_{n,\ell}(t)-\u_{n,i}(t))\cdot \dot \u_{n,\ell}(t)\\
=&\frac{\kappa }{2}\sum_{j=0}^{k_n}\sum_{\x_{i,n}\in T_j\cap\mathcal{N}_n}
\sum_{k\in \tau_{n,i}}\sum_{\x_{n,\ell}\in T_k\cap \mathcal{N}_n}
\frac{\mathcal{H}^2(T_j)\mathcal{H}^2(T_k)}{9}
\frac{|\u_{n,\ell}(t)-\u_{n,i}(t)|^{p-2}}{d_{\M}(\x_{n,\ell},\x_{n,i})^{2+\alpha p}}(\u_{n,\ell}(t)-\u_{n,i}(t))\cdot( \dot \u_{n,i}(t)-\dot \u_{n,\ell}(t))\\
=&-\frac{d}{dt}\left(\frac{\kappa }{2p}\sum_{j=0}^{k_n}\sum_{\x_{i,n}\in T_j\cap\mathcal{N}_n}
\sum_{k\in \tau_{n,i}}\sum_{\x_{n,\ell}\in T_k\cap \mathcal{N}_n}
\frac{\mathcal{H}^2(T_j)\mathcal{H}^2(T_k)}{9}
\frac{|\u_{n,\ell}(t)-\u_{n,i}(t)|^{p}}{d_{\M}(\x_{n,\ell},\x_{n,i})^{2+\alpha p}}\right).
\end{align*}}
{To estimate $I_2$ we use the discrete H\"older inequality, so that, we get 
\begin{align*}
I_2=&\sum_{j=0}^{k_n}\sum_{\x_{i,n}\in T_j\cap\mathcal{N}_n}\frac{\mathcal{H}^2(T_j)}{3}\b_n(\x_{n,i},t)\cdot \dot \u_{n,i}(t)\\
\le &\sqrt{\frac{2}{\rho}}\left(\sum_{j=0}^{k_n}\sum_{\x_{i,n}\in T_j\cap\mathcal{N}_n}\frac{\mathcal{H}^2(T_j)}{3}|\b_n(\x_{n,i},t)|^2\right)^{1/2}
\left(\rho \sum_{j=0}^{k_n}\sum_{\x_{i,n}\in T_j\cap\mathcal{N}_n}\frac{\mathcal{H}^2(T_j)}{3}\frac{|\dot \u_{n,i}(t)^2}{2}\right)^{1/2}\\
\le &\sqrt{\frac{2}{\rho}}\left(\sum_{j=0}^{k_n}\sum_{\x_{i,n}\in T_j\cap\mathcal{N}_n}\frac{\mathcal{H}^2(T_j)}{3}|\b_n(\x_{n,i},t)|^2\right)^{1/2}
\left(E_n(t)\right)^{1/2}.
\end{align*}}
Therefore, \eqref{eq:en0} gives
{\begin{equation*}
\dot E_n(t)\le \sqrt{\frac{2}{\rho}}\left(\sum_{j=0}^{k_n}\sum_{\x_{i,n}\in T_j\cap\mathcal{N}_n}\frac{\mathcal{H}^2(T_j)}{3}|\b_n(\x_{n,i},t)|^2\right)^{1/2}\left(E_n(t)\right)^{1/2},
\end{equation*}}
that is 
{\begin{equation*}
\frac{d}{dt} \sqrt{E_n(t)}\le \frac{1}{\sqrt{2\rho}}\left(\sum_{j=0}^{k_n}\sum_{\x_{i,n}\in T_j\cap\mathcal{N}_n}\frac{\mathcal{H}^2(T_j)}{3}|\b_n(\x_{n,i},t)|^2\right)^{1/2}.
\end{equation*}}
Integrating over $(0,t)$ we conclude
{\begin{equation*}
 \sqrt{E_n(t)}\le  \sqrt{E_n(0)}+\frac{1}{\sqrt{2\rho}}\int_0^t\left(\sum_{j=0}^{k_n}\sum_{\x_{i,n}\in T_j\cap\mathcal{N}_n}\frac{\mathcal{H}^2(T_j)}{3}|\b_n(\x_{n,i},t)|^2\right)^{1/2}dt.
\end{equation*}}
\end{proof}

\begin{lemma}[{\bf Discrete $L^2$ estimate}]
\label{lm:l2est}
Let $T\ge0$ be given and $\{\u_{n,i}\}$ be the solution of \eqref{eq:CPnum}. The  estimate 
{\begin{equation}
\label{eq:l2est}
\left(\sum_{j=0}^{k_n}\sum_{\x_{i,n}\in T_j\cap\mathcal{N}_n}\frac{\mathcal{H}^2(T_j)}{3}|\u_{n,i}(t)|^2\right)^{1/2}\le C_T(1+t)
\end{equation}}
holds true for every $0<t\le T$.
\end{lemma}

\begin{proof}
Let $t\in(0,T)$ be given.
We observe that
{\begin{equation}
\label{eq:l2.1}
\begin{split}
|\u_{n,i}(t)|\le&|\u_{0,n,i}|+\int_0^t |\dot \u_{n,i}(s)|\,ds\\
\le&|\u_{0,n,i}|+\sqrt{t}\sqrt{\int_0^t |\dot \u_{n,i}(s)|^2\,ds}.\end{split}
\end{equation}}
Taking the square of both sides of~\eqref{eq:l2.1}
and summing over $i$, we have
{\begin{align*}
\sum_{j=0}^{k_n}\sum_{\x_{i,n}\in T_j\cap\mathcal{N}_n}\frac{\mathcal{H}^2(T_j)}{3}|\u_{n,i}(t)|^2\le&\>
2\sum_{j=0}^{k_n}\sum_{\x_{i,n}\in T_j\cap\mathcal{N}_n}\frac{\mathcal{H}^2(T_j)}{3}|\u_{0,n,i}|^2+2t\int_0^t 
\sum_{j=0}^{k_n}\sum_{\x_{i,n}\in T_j\cap\mathcal{N}_n}\frac{\mathcal{H}^2(T_j)}{3}|\dot \u_{n,i}(s)|^2\,ds\\
\le&\>
{2}\sum_{j=0}^{k_n}\sum_{\x_{i,n}\in T_j\cap\mathcal{N}_n}\frac{\mathcal{H}^2(T_j)}{3}|\u_{0,n,i}|^2+
2t^2\sup_{t\in(0,T)}\sum_{j=0}^{k_n}\sum_{\x_{i,n}\in T_j\cap\mathcal{N}_n}\frac{\mathcal{H}^2(T_j)}{3}|\dot \u_{n,i}(t)|^2.
\end{align*}}
Therefore, the desired
estimate follows from~\eqref{eq:assinitial_n} and~\eqref{eq:enest}.
\end{proof}

\textcolor{blue}{
\begin{lemma}
\label{lm:Sobenest}
The sequence $\{\u_n\}_n$ defined in \eqref{eq:franco} is  bounded  in $L^\infty(0,T;\mathcal{W})$, for very $T>0$.
\end{lemma}}

\begin{proof}
\textcolor{blue}{ We begin by proving that 
\begin{equation}
\label{eq:l2affine}
\text{$\{\u_n\}_n$ is a bounded sequence in $L^\infty(0,T;L^2(\M;\R^3))$, for very $T>0$.}
\end{equation}
Due to the definition of $\u_n$ we know that
\begin{equation}
\label{eq:linear}
 |\u_n(\x,t)|\le |\u_{n,i}(t)|+\max_{j\not =i}\frac{ |\u_{n,i}(t)-\u_{n,j}(t)|}{|\x_{n,i}-\x_{n,j}|}|\x_{n,i}-\x|,\qquad \forall\,\x\in T_i.
\end{equation}
Thanks to \eqref{eq:reg}, and \eqref{eq:locqu}  we have 
\begin{align*}
\norm{ \u_n(t,\cdot)}_{L^2(\M;\R^3)}^2=&\int_\M | \u_n(\x,t)|^2d\mathcal{H}^2(\x)
\le c\sum_{i=0}^{k_n}\int_{T_i} | \u_n(\x,t)|^2d\mathcal{H}^2(\x)\\
\le&c\sum_{i=0}^{k_n}\left(\int_{T_i} |\u_{n,i}|^2d\mathcal{H}^2(\x)+\max_{j\not =i}\int_{T_i}\frac{ |\u_{n,i}(t)-\u_{n,j}(t)|^2}{|\x_{n,i}-\x_{n,j}|^2}|\x_{n,i}-\x|^2d\mathcal{H}^2(\x)\right)\\
\le&c\sum_{i=0}^{k_n}\mathcal{H}^2(T_i)\left( |\u_{n,i}(t)|^2+\max_{j\not =i}\frac{ |\u_{n,i}(t)-\u_{n,j}(t)|^2}{\min\limits_\ell\rho_{T_\ell}^2}h_{T_i}^2\right)\\
\le&c\sum_{i=0}^{k_n}\mathcal{H}^2(T_i)\left( |\u_{n,i}(t)|^2+\max_{j\not =i}\frac{ |\u_{n,i}(t)|^2+|\u_{n,j}(t)|^2}{\min\limits_\ell\rho_{T_\ell}^2}h_{T_i}^2\right)\\
\le&c\sum_{i=0}^{k_n}\mathcal{H}^2(T_i) |\u_{n,i}(t)|^2.
\end{align*}
Therefore, \eqref{eq:l2affine} follows from Lemma \ref{lm:l2est}.}

\textcolor{blue}{We continue by estimating the second term in \eqref{eq:norm}.
In order to help the readability of the computations we consider only the case $\delta=\infty$, that is not restrictive in light of the equivalence between $\mathcal{W}$ and $W^{\alpha,p}$.
We split the set of indexes as follows
\begin{equation*}
\{0,...,k_n\}\times \{0,...,k_n\}=\mathcal{I}_{1,n}\cup\mathcal{I}_{2,n}\cup\mathcal{I}_{3,n},
\end{equation*}
where
\begin{align*}
\mathcal{J}_{1,n}=&\big\{(i,j):i=j\big\},\\
\mathcal{J}_{2,n}=&\big\{(i,j):i \not=j,\,\partial T_i\cap \partial T_j\not=\emptyset\big\},\\
\mathcal{J}_{3,n}=&\big\{(i,j):i \not=j,\,\partial T_i\cap \partial T_j=\emptyset\big\}.
\end{align*}
Accordingly we split the second term in \eqref{eq:norm} as follows
\begin{align}
\int_{\M}\int_{\M}&\frac{|\u_n(\x,t)-\u_n(\x',t)|^p}{d_{\M}(\x,\x')^{2+\alpha p}}\,d\mathcal{H}^2(\x')\,d\mathcal{H}^2(\x)\notag\\
=&\sum_{i,j=0}^{k_n}\int_{T_i}\int_{T_j}\frac{|\u_n(\x,t)-\u_n(\x',t)|^p}{d_{\M}(\x,\x')^{2+\alpha p}}\,d\mathcal{H}^2(\x')\,d\mathcal{H}^2(\x)\notag\\
=&\sum_{(i,j)\in \mathcal{J}_{1,n}}\int_{T_i}\int_{T_j}\frac{|\u_n(\x,t)-\u_n(\x',t)|^p}{d_{\M}(\x,\x')^{2+\alpha p}}\,d\mathcal{H}^2(\x')\,d\mathcal{H}^2(\x)\label{eq:J_1}\\
&+\sum_{(i,j)\in \mathcal{J}_{2,n}}\int_{T_i}\int_{T_j}\frac{|\u_n(\x,t)-\u_n(\x',t)|^p}{d_{\M}(\x,\x')^{2+\alpha p}}\,d\mathcal{H}^2(\x')\,d\mathcal{H}^2(\x)\label{eq:J_2}\\
&+\sum_{(i,j)\in \mathcal{J}_{3,n}}\int_{T_i}\int_{T_j}\frac{|\u_n(\x,t)-\u_n(\x',t)|^p}{d_{\M}(\x,\x')^{2+\alpha p}}\,d\mathcal{H}^2(\x')\,d\mathcal{H}^2(\x).\label{eq:J_3}
\end{align}
We estimate separately the last three  terms. }

\textcolor{blue}{We begin with \eqref{eq:J_1}.
Due to the definition of $\u_n$ we know that
\begin{equation}
\label{eq:linear1}
 |\u_n(\x,t)-\u_n(\x',t)|\le c\sum_{\ell\not =i}\frac{ |\u_{n,i}(t)-\u_{n,\ell}(t)|}{|\x_{n,i}-\x_{n,\ell}|}d_{\M}(\x,\x'),\qquad \x,\,\x'\in T_i.
\end{equation}
Thanks to \eqref{eq:reg}, and \eqref{eq:locqu} we have 
\begin{align*}
\sum_{(i,j)\in \mathcal{J}_{1,n}}&\int_{T_i}\int_{T_j}\frac{|\u_n(\x,t)-\u_n(\x',t)|^p}{d_{\M}(\x,\x')^{2+\alpha p}}\,d\mathcal{H}^2(\x')\,d\mathcal{H}^2(\x)\\
\le &\sum_{i=0}^{k_n}\int_{T_i}\int_{T_i}\frac{|\u_n(\x,t)-\u_n(\x',t)|^p}{d_{\M}(\x,\x')^{2+\alpha p}}\,d\mathcal{H}^2(\x')\,d\mathcal{H}^2(\x)\\
\le &c\sum_{i=0,\,\ell\not= i}^{k_n}\frac{ |\u_{n,i}(t)-\u_{n,\ell}(t)|^p}{|\x_{n,i}-\x_{n,\ell}|^p}   \int_{T_i}\int_{T_i}\frac{d\mathcal{H}^2(\x')\,d\mathcal{H}^2(\x)}{d_{\M}(\x,\x')^{2-(1-\alpha) p}}\\
\le &c\sum_{i=0,\,\ell\not= i}^{k_n} \frac{ |\u_{n,i}(t)-\u_{n,\ell}(t)|^p}{|\x_{n,i}-\x_{n,\ell}|^{2+\alpha p}} |\x_{n,i}-\x_{n,\ell}|^{2-(1-\alpha) p}  \int_{B_{h_{T_i}}(\x_{n,i})}\int_{B_{h_{T_i}}(\x_{n,i})}
\frac{d\mathcal{H}^2(\x')\,d\mathcal{H}^2(\x)}{d_{\M}(\x,\x')^{2-(1-\alpha) p}}\\
= &c\sum_{i=0,\,\ell\not= i}^{k_n}\frac{ |\u_{n,i}(t)-\u_{n,\ell}(t)|^p}{|\x_{n,i}-\x_{n,\ell}|^{2+\alpha p}} |\x_{n,i}-\x_{n,\ell}|^{2-(1-\alpha) p} h_{T_i}^{2+(1-\alpha) p} \int_{B_{1}(\x_{n,i})}\int_{B_{1}(\x_{n,i})}
\frac{d\mathcal{H}^2(\x')\,d\mathcal{H}^2(\x)}{d_{\M}(\x,\x')^{2-(1-\alpha) p}}\\
\le  &c\sum_{i=0,\,\ell\not= i}^{k_n}\frac{ |\u_{n,i}(t)-\u_{n,\ell}(t)|^p}{|\x_{n,i}-\x_{n,\ell}|^{2+\alpha p}} h_{T_i}^4\int_{B_{1}(\x_{n,i})}\int_{B_{1}(\x_{n,i})}
\frac{d\mathcal{H}^2(\x')\,d\mathcal{H}^2(\x)}{d_{\M}(\x,\x')^{2-(1-\alpha) p}}\\
\le  &c\sum_{i=0,\,\ell\not= i}^{k_n}\frac{ |\u_{n,i}(t)-\u_{n,\ell}(t)|^p}{|\x_{n,i}-\x_{n,\ell}|^{2+\alpha p}} \mathcal{H}^2(T_i) \mathcal{H}^2(T_\ell),
\end{align*}
where we used the facts
\begin{equation*}
2-(1-\alpha) p<2,\qquad |\x_{n,i}-\x_{n,\ell}|^{2-(1-\alpha) p} \le c\begin{cases}
 \max\limits_j h_{T_j}^{2-(1-\alpha) p}, &\text{if $2-(1-\alpha) p\ge0$},\\
 \min\limits_j \rho_{T_j}^{2-(1-\alpha) p}, &\text{if $2-(1-\alpha) p<0$}.
 \end{cases}
\end{equation*}
Therefore, the term in \eqref{eq:J_1} is bounded due to Lemma \ref{lm:enest}.}

\textcolor{blue}{We continue with the term \eqref{eq:J_2}.
Using again the definition of $\u_n$ we know that
\begin{align}
\label{eq:linear3}
 |\u_n(\x,t)-\u_n(\x',t)|\le& c\left(\sum_{\ell\not =i}\frac{ |\u_{n,i}(t)-\u_{n,\ell}(t)|}{|\x_{n,i}-\x_{n,\ell}|}+\sum_{\ell\not =j}\frac{ |\u_{n,j}(t)-\u_{n,\ell}(t)|}{|\x_{n,j}-\x_{n,\ell}|}\right)d_{\M}(\x,\x'),
 \quad 
(\x,\,\x')\in T_i\times T_j,\,(i,j)\in\mathcal{J}_{n,2}.\notag
\end{align}
Thanks to \eqref{eq:reg}, and \eqref{eq:locqu} we have \begin{align*}
\sum_{(i,j)\in \mathcal{J}_{2,n}}&\int_{T_i}\int_{T_j}\frac{|\u_n(\x,t)-\u_n(\x',t)|^p}{d_{\M}(\x,\x')^{2+\alpha p}}\,d\mathcal{H}^2(\x')\,d\mathcal{H}^2(\x)\\
\le &c\!\!\sum_{(i,j)\in \mathcal{J}_{2,n}}\int_{T_i}\int_{T_j}\left(\sum_{\ell\not =i}\frac{ |\u_{n,i}(t)-\u_{n,\ell}(t)|^p}{|\x_{n,i}-\x_{n,\ell}|^p}
+\sum_{\ell\not =j}\frac{ |\u_{n,j}(t)-\u_{n,\ell}(t)|^p}{|\x_{n,j}-\x_{n,\ell}|^p}\right)\frac{d\mathcal{H}^2(\x')\,d\mathcal{H}^2(\x)}{d_{\M}(\x,\x')^{2-(1-\alpha) p}}\\
\le &c\!\!\sum_{(i,j)\in \mathcal{J}_{2,n}}\left(\sum_{\ell\not =i}\frac{ |\u_{n,i}(t)-\u_{n,\ell}(t)|^p}{|\x_{n,i}-\x_{n,\ell}|^{2+\alpha p}}
+\sum_{\ell\not =j}\frac{ |\u_{n,j}(t)-\u_{n,\ell}(t)|^p}{|\x_{n,j}-\x_{n,\ell}|^{2+\alpha p}}\right)\max_q h_{T_q}^{2-(1-\alpha) p} \int_{T_i}\int_{T_j}\frac{d\mathcal{H}^2(\x')\,d\mathcal{H}^2(\x)}{d_{\M}(\x,\x')^{2-(1-\alpha) p}}\\
\le &c\!\!\sum_{(i,j)\in \mathcal{J}_{2,n}}\left(\sum_{\ell\not =i}\frac{ |\u_{n,i}(t)-\u_{n,\ell}(t)|^p}{|\x_{n,i}-\x_{n,\ell}|^{2+\alpha p}}
+\sum_{\ell\not =j}\frac{ |\u_{n,j}(t)-\u_{n,\ell}(t)|^p}{|\x_{n,j}-\x_{n,\ell}|^{2+\alpha p}}\right)\max_q h_{T_q}^{2-(1-\alpha) p}
 \int_{B_{h_{T_i}}(\x_{n,i})}\int_{B_{h_{T_j}}(\x_{n,j})}\frac{d\mathcal{H}^2(\x')\,d\mathcal{H}^2(\x)}{d_{\M}(\x,\x')^{2-(1-\alpha) p}}\\
 \le &c\!\!\sum_{(i,j)\in \mathcal{J}_{2,n}}\left(\sum_{\ell\not =i}\frac{ |\u_{n,i}(t)-\u_{n,\ell}(t)|^p}{|\x_{n,i}-\x_{n,\ell}|^{2+\alpha p}}
+\sum_{\ell\not =j}\frac{ |\u_{n,j}(t)-\u_{n,\ell}(t)|^p}{|\x_{n,j}-\x_{n,\ell}|^{2+\alpha p}}\right)\max_q h_{T_q}^{4}
 \int_{B_{1}(\x_{n,i})}\int_{B_{1}(\x_{n,j})}\frac{d\mathcal{H}^2(\x')\,d\mathcal{H}^2(\x)}{d_{\M}(\x,\x')^{2-(1-\alpha) p}}\\
\le  &c\!\!\sum_{i=0,\,\ell\not= i}^{k_n}\frac{ |\u_{n,i}(t)-\u_{n,\ell}(t)|^p}{|\x_{n,i}-\x_{n,\ell}|^{2+\alpha p}} \mathcal{H}^2(T_i) \mathcal{H}^2(T_\ell),
\end{align*}
where we used again the inequality
\begin{equation*}
2-(1-\alpha) p<2.
\end{equation*}
Therefore, also the term in \eqref{eq:J_2} is bounded due to Lemma \ref{lm:enest}.}

\textcolor{blue}{We conclude with the term \eqref{eq:J_3}. Using again the definition of $\u_n$ we know that
\begin{equation}
\label{eq:linear3}
\begin{cases}
 |\u_n(\x,t)-\u_n(\x',t)|\le |\u_{n,i}(t)-\u_{n,j}(t)|&{}\\
 \qquad\qquad\qquad\displaystyle+\sum_{\ell\not =i}\frac{ |\u_{n,i}(t)-\u_{n,\ell}(t)|}{|\x_{n,i}-\x_{n,\ell}|}h_{T_i}&{}\\
 \qquad\qquad\qquad\displaystyle +\sum_{\ell\not =j}\frac{ |\u_{n,j}(t)-\u_{n,\ell}(t)|}{|\x_{n,j}-\x_{n,\ell}|}h_{T_j},&{}\\
 d_{\M}(\x,\x')\ge\min\limits_\ell \rho_{T_\ell},& {}\end{cases}
 \quad 
(\x,\,\x')\in T_i\times T_j,\,(i,j)\in\mathcal{J}_{n,3}.
\end{equation}
Thanks to \eqref{eq:reg}, and \eqref{eq:locqu} we have 
\begin{align*}
\sum_{(i,j)\in \mathcal{J}_{3,n}}&\int_{T_i}\int_{T_j}\frac{|\u_n(\x,t)-\u_n(\x',t)|^p}{d_{\M}(\x,\x')^{2+\alpha p}}\,d\mathcal{H}^2(\x')\,d\mathcal{H}^2(\x)\\
\le &\frac{c}{\min\limits_\ell \rho_{T_\ell}^{2+\alpha p}}
\sum_{(i,j)\in \mathcal{J}_{3,n}}
\left( |\u_{n,i}(t)-\u_{n,j}(t)|^p+\sum_{\ell\not =i}\frac{ |\u_{n,i}(t)-\u_{n,\ell}(t)|^p}{|\x_{n,i}-\x_{n,\ell}|^p}h_{T_i}^p+\right.\\
&\qquad\qquad\qquad\qquad\left. +\sum_{\ell\not =j}\frac{ |\u_{n,j}(t)-\u_{n,\ell}(t)|^p}{|\x_{n,j}-\x_{n,\ell}|^p}h_{T_j}^p\right)\mathcal{H}^2(T_i) \mathcal{H}^2(T_j)\\
\le &\frac{c}{\min\limits_\ell \rho_{T_\ell}^{2+\alpha p}}
\sum_{(i,j)\in \mathcal{J}_{3,n}}
\left( \frac{|\u_{n,i}(t)-\u_{n,j}(t)|^p}{|\x_{n,i}-\x_{n,j}|^{2+\alpha p}}|\x_{n,i}-\x_{n,j}|^{2+\alpha p}\right.\\
&\qquad\qquad\qquad\qquad\left. +\sum_{\ell\not =i}\frac{ |\u_{n,i}(t)-\u_{n,\ell}(t)|^p}{|\x_{n,i}-\x_{n,\ell}|^{2+\alpha p}}|\x_{n,i}-\x_{n,\ell}|^{2-(1-\alpha) p} h_{T_i}^p+\right.\\
&\qquad\qquad\qquad\qquad\left. +\sum_{\ell\not =j}\frac{ |\u_{n,j}(t)-\u_{n,\ell}(t)|^p}{|\x_{n,j}-\x_{n,\ell}|^{2+\alpha p}}|\x_{n,j}-\x_{n,\ell}|^{2-(1-\alpha) p}h_{T_j}^p\right)\mathcal{H}^2(T_i) \mathcal{H}^2(T_j)\\
\le   &c\sum_{i=0,\,\ell\not= i}^{k_n} \frac{ |\u_{n,i}(t)-\u_{n,\ell}(t)|^p}{|\x_{n,i}-\x_{n,\ell}|^{2+\alpha p}} \mathcal{H}^2(T_i) \mathcal{H}^2(T_\ell),\end{align*}}
\end{proof}

\textcolor{blue}{Arguing as in \eqref{eq:l2affine} we can prove also the following lemma.}

\textcolor{blue}{
\begin{lemma}
\label{lm:taff}
Let  $\{\u_n\}_n$ be the sequence defined in \eqref{eq:franco}.
Then $\{\pt \u_n\}_n$ is a bounded sequence in $L^\infty(0,T;L^2(\M;\R^3))$, for very $T>0$.
\end{lemma}}


\begin{proof}[Proof of Proposition \ref{lm:exists}]

Thanks to the definition of $\u_n$ and Lemmas \textcolor{blue}{\ref{lm:Sobenest} and \ref{lm:taff}} we have that
\begin{align*}
&\text{$\{\pt \u_n\}_n$ is a bounded sequence in $L^\infty(0,T;L^2(\M;\R^3))$, for very $T>0$},\\
&\text{$\{\u_n\}_n$ is a bounded sequence in $L^\infty(0,T;\mathcal{W})$, for very $T>0$}.
\end{align*}
Therefore, by Lemma \ref{lm:compembed}, we have that there exist a subsequence $\{\u_{n_h}\}_h$ and a function $\u \in \mathcal{X}$ such that \eqref{eq:exist1} holds true.

We have to prove \eqref{eq:exist2}. For the sake of notation simplicity, we  label the sequence with $n$ and not $n_h$.
Let $\vfi\in C^\infty({\M\times\R};\R^3)$ be a test function  with compact support. Multiplying \eqref{eq:CPnum} by $\vfi(\x_{n,i},t)$,
integrating over $t$ and summing over $i$ we get
{\begin{align*}
&\underbrace{\int_0^\infty \sum_{j=0}^{k_n}\sum_{\x_{i,n}\in T_j\cap\mathcal{N}_n}\frac{\mathcal{H}^2(T_j)}{3}\ddot \u_{n,i}(t)\vfi(\x_{n,i},t)dt}_{\mathcal{D}_n}\\
&\qquad=\underbrace{\int_0^\infty \sum_{j=0}^{k_n}\sum_{\x_{i,n}\in T_j\cap\mathcal{N}_n}\frac{\mathcal{H}^2(T_j)}{3}(K_{n}\u_n(\cdot,t))(\x_{n,i})\vfi(\x_{n,i},t)dt}_{\mathcal{E}_n}\\
&\qquad\qquad+\underbrace{\int_0^\infty \sum_{j=0}^{k_n}\sum_{\x_{i,n}\in T_j\cap\mathcal{N}_n}\frac{\mathcal{H}^2(T_j)}{3}\b_n(\x_{n,i},t)\vfi(\x_{n,i},t)dt.}_{\mathcal{F}_n}
\end{align*}}

Thanks to \eqref{eq:assinitial_n} and \eqref{eq:exist1} we have
{\begin{align*}
\mathcal{D}_n=&\int_0^\infty \sum_{j=0}^{k_n}\sum_{\x_{i,n}\in T_j\cap\mathcal{N}_n}\frac{\mathcal{H}^2(T_j)}{3} \u_{n,i}(t)\ptt\vfi(\x_{n,i},t)dt\\
&-2\int_0^\infty \sum_{j=0}^{k_n}\sum_{\x_{i,n}\in T_j\cap\mathcal{N}_n}\frac{\mathcal{H}^2(T_j)}{3} \vv_{0,n}(\x_{n,i})\vfi(\x_{n,i},t)dt+
2\int_0^\infty \sum_{j=0}^{k_n}\sum_{\x_{i,n}\in T_j\cap\mathcal{N}_n}\frac{\mathcal{H}^2(T_j)}{3}  \vv_{0,n}(\x_{n,i})\vfi(\x_{n,i},t)dt\to\\
\to&\int_0^\infty\!\!\!\!\int_{\M} \u(\x,t)\ptt\vfi(\x,t)\,dt\,d\mathcal{H}^2(\x)-\int_{\M} \vv_0(\x)\vfi(\x,0)d\mathcal{H}^2(\x)+\int_{\M} \u_0(\x)\pt\vfi(\x,0)d\mathcal{H}^2(\x),\\
\mathcal{F}_n\to&\int_0^\infty\!\!\!\!\int_{\M} \b(\x,t)\vfi(\x,t) \,dt\,d\mathcal{H}^2(\x).
\end{align*}}
We have to prove that
\begin{equation}
\label{eq:sol-weak-alpha1}
\mathcal{E}_n\to\int_0^\infty\!\!\!\!\int_{\M}(K \u(\cdot,t))(\x)\vfi (\x,t) \,dt\,d\mathcal{H}^2(\x).
\end{equation}
We split $\mathcal{E}_n$ as follows
{\begin{align*}
\mathcal{E}_n=&\underbrace{\int_0^\infty \sum_{j=0}^{k_n}\sum_{\x_{i,n}\in T_j\cap\mathcal{N}_n}\frac{\mathcal{H}^2(T_j)}{3}(K_{n}\u_n(\cdot,t)-K\u_n(\cdot,t))(\x_{n,i})\vfi(\x_{n,i},t)dt}_{\mathcal{E}_{n,1}}\\
&\underbrace{+\int_0^\infty \sum_{j=0}^{k_n}\sum_{\x_{i,n}\in T_j\cap\mathcal{N}_n}\frac{\mathcal{H}^2(T_j)}{3}(K\u_n(\cdot,t))(\x_{n,i})\vfi(\x_{n,i},t)dt
-\int_0^\infty\!\!\!\!\int_{\M}(K \u_n(\cdot,t))(\x)\vfi (\x,t) \,dt\,d\mathcal{H}^2(\x)}_{\mathcal{E}_{n,2}}\\
&\underbrace{+\int_0^\infty\!\!\!\!\int_{\M}(K \u_n(\cdot,t))(\x)\vfi (\x,t) \,dt\,d\mathcal{H}^2(\x).}_{\mathcal{E}_{n,3}}
\end{align*}}
We observe that, arguing as in Lemma \ref{lm:enest}
\begin{align*}
\int_0^\infty\!\!\!\!\int_{\M}&(K \u_n(\cdot,t))(\x)\vfi (\x,t) \,dt\,d\mathcal{H}^2(\x)\\
=&\kappa\int_0^\infty\!\!\!\!\int_{\M}\int_{\M\cap B_\delta(\x)}
\frac{|\u(\x',t)-\u(\x,t)|^{p-2}}{d_{\M}(\x',\x)^{2+\alpha p}}(\u(\x',t)-\u(\x,t))\vfi (\x,t) \,dt\,d\mathcal{H}^2(\x)d\mathcal{H}^2(\x')\\
=&\frac{\kappa}{2}\int_0^\infty\!\!\!\!\int_{\M}\int_{\M\cap B_\delta(\x)}
\frac{|\u(\x',t)-\u(\x,t)|^{p-2}}{d_{\M}(\x',\x)^{2+\alpha p}}(\u(\x',t)-\u(\x,t))(\vfi (\x,t)-\vfi (\x',t)) \,dt\,d\mathcal{H}^2(\x)d\mathcal{H}^2(\x')\\
\le&\frac{\kappa}{2}\underbrace{\left(\int_0^\infty\!\!\!\!\int_{\M}\int_{\M\cap B_\delta(\x)}
\frac{|\u(\x',t)-\u(\x,t)|^{p}}{d_{\M}(\x',\x)^{2+\alpha p}}\,dt\,d\mathcal{H}^2(\x)d\mathcal{H}^2(\x')\right)^{\frac{p-1}{p}}}_{\text{bounded due to Lemma \ref{lm:enest}}}
\times\\
&\qquad \times\underbrace{\left(\int_0^\infty\!\!\!\!\int_{\M}\int_{\M\cap B_\delta(\x)}
\frac{|\vfi(\x',t)-\vfi(\x,t)|^{p}}{d_{\M}(\x',\x)^{p}}
\frac{1}{d_{\M}(\x',\x)^{2+\alpha p-p}}\,dt\,d\mathcal{H}^2(\x)d\mathcal{H}^2(\x')\right)^{\frac{1}{p}}}_{\text{bounded because $\vfi\in C^\infty$
and $2+\alpha p-p<2$}},
\end{align*}
therefore standard arguments guarantee
\begin{equation*}
\mathcal{E}_{n,1}\to0,\qquad \mathcal{E}_{n,2}\to0.
\end{equation*}
Moreover, thanks to Lemma \ref{lm:compembed} and \eqref{eq:exist1}
\begin{equation*}
\mathcal{E}_{n,3}\to\int_0^\infty\!\!\!\!\int_{\M}(K \u(\cdot,t))(\x)\vfi (\x,t) \,dt\,d\mathcal{H}^2(\x).
\end{equation*}

Finally, \eqref{eq:sol-diss-alpha} follows from Lemma \ref{lm:enest}, \eqref{eq:exist1}, and \eqref{eq:assinitial_n}.
\end{proof}

\begin{proof}[Proof of Theorem \ref{th:main}]
{\bf (A)} and {\bf (B)} follow from Lemma \ref{lm:exists},
{\bf (D)} can be proved arguing as in \cite[Section 4]{CDMV}, and {\bf (C)} follows from {\bf (D)}.
\end{proof}

\section{\bf Numerical scheme and computational strategy}
\label{numscheme}

{In this section the numerical procedure for the approximation of the solution of the Cauchy problem \eqref{eq:CP} is detailed. Firstly, its semi-discrete approximation is given through the projection of the 2D manifold $\mathcal{M}$ onto the triangular mesh $M$. The problem of computing the geodesic distance on a mesh is discussed and the Dijkstra algorithm is proposed as a routing tool on the latter. Such algorithm is used for defining the discrete function $d_M(\cdot , \cdot)$ returning the distance between two points on $M$ as total length of the shortest path connecting them. Lastly, the time integration of \eqref{eq:CP} is performed through a P-(EC)$^k$ formulation of the $\beta$--Newmark scheme. Note that, since the kernel ${\bf K}(\u(\x))$ is nonlinear, the magnitude of the structural internal response to any external stress or initial condition is hard to predict. Moreover, the nonlocal nature of ${\bf K}(\u(\x))$ may imply interactions and competitions between different spatial scales that may cause the repentine transfer of large gradient of energy between them. Due to these reasons, the authors choose to adopt an implicit P-(EC)$^k$ scheme since the stability of an explicit scheme can be affected up to unfeasible time-steps. In this context, the $\beta$--Newmark scheme is extremely suitable being stable regardless of the size of the time step (i.e. $\frac{1}{2}\leq\gamma\leq 2\beta$).}

\subsection{Spatial discretization}

The evolution  problem  is computed on a triangular mesh $M$ composed by $nv$ nodes connected by $ne$ edges and conforming the two dimensional manifold $\M$. The semi-discrete approximation of the Cauchy problem \eqref{eq:CP} for $i \in {1, 2, 3, ..., nv}: \x_i \in M$ reads:
\begin{equation}
\begin{cases}
\rho_i \ddot{\u}_i(t) = {\bf K}_i(\u_i)(t) + {\bf b}_i(t),\\
\displaystyle{\bf K}_i(\u_i)(t) = \sum_{j\in M\cap B^i_\delta} k_{i,j} \left( \frac{|\u_i(t) - \u_j(t)|^{p-2}}{d_{M}(\x_i,\x_j)^{2 + \alpha p}} (\u_i(t) - \u_j(t)) \right)\Delta A_j,\\
\u_i(0)=\u_{i,0}, \\
\dot{\u}_i(0) = {\bf v}_{i,0}, \\
\end{cases}
\label{semiD}
\end{equation}
with $\rho_i$ the density assigned to the node i; $B^j_\delta$ the ball by radius $\delta$ centered in $x_i$; $k_{i,j}$ the specific pairwise elastic modulus; $d_{M}(\cdot,\cdot)$ the approximation of the geodesic distance $d_{\M}(\cdot,\cdot)$ computed as a polygonal in $M$; $\Delta A_j$ the portion of the triangular cell area assigned to j; and $\u_{i,0}$ and ${\bf v}_{i,0}$ the initial displacement and velocity for i.

\subsection{Dijkstra algorithm for computing geodesic distances on $M$}

The polygonal approximating the geodesic distance function $d_{\M}(\cdot,\cdot)$ is obtained by computing the discrete shortest path on $M$ connecting two given points (see {\bf Figure.~\ref{distance}a}). A path on $M$ is defined as the collection of edges traveled from the starting to the ending point of the path. So that, the shortest path is meant in terms of the sum of the edge lengths for connecting the two selected points. Indeed, edge lengths are computed as Euclidean distances and the total length of the polygonal connecting two points is meant as the approximation of $d_{\M}(\cdot,\cdot)$ on $M$, $d_{M}(\cdot,\cdot)$. To find the shortest path on $M$, the Dijkstra algorithm is adopted. 
\begin{figure}
\centering
\includegraphics[scale=0.3]{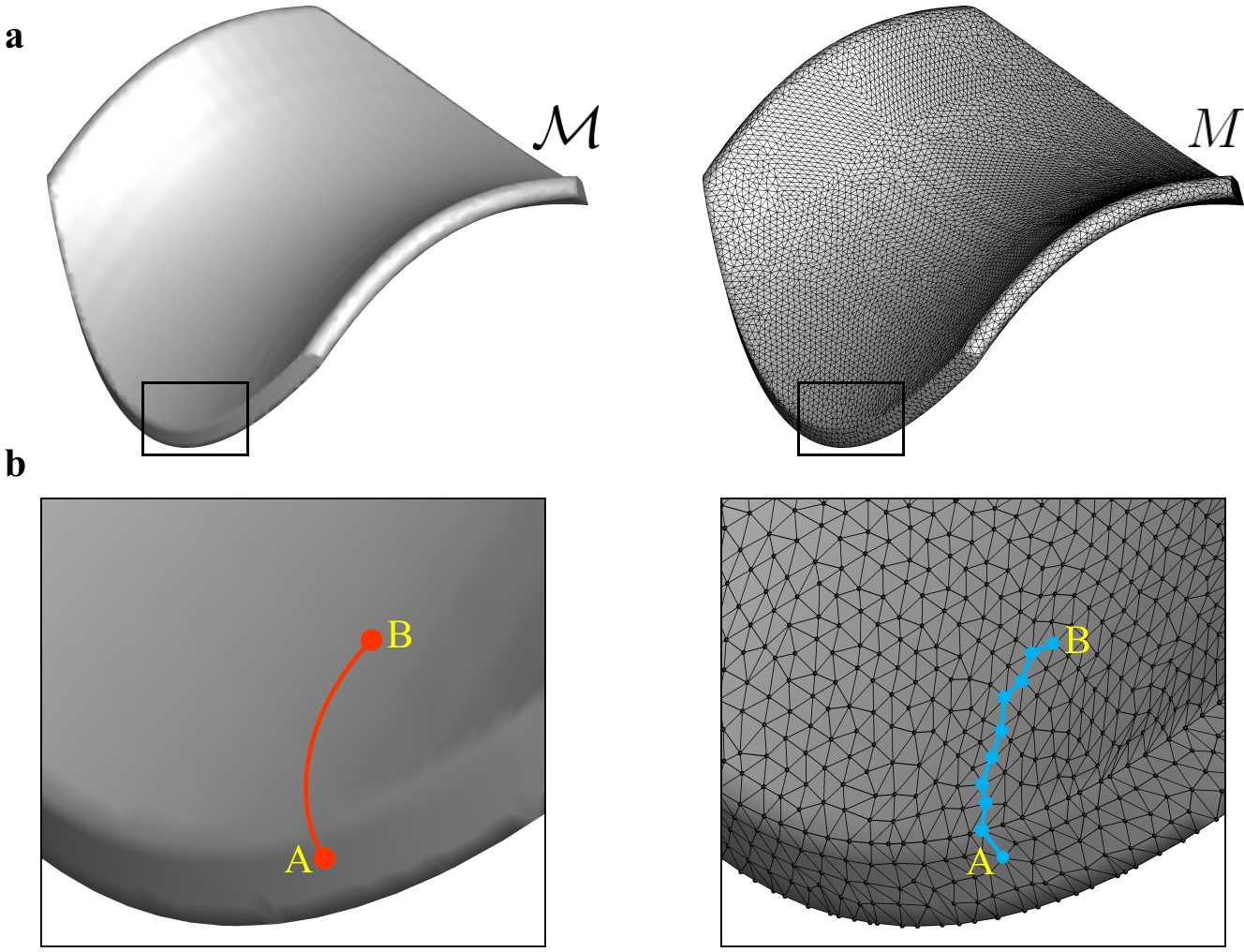}
\caption{{\bf Computational mesh and spatial discretization.} {\bf a.} Approximation of the manifold $\M$ into the conforming mesh $M$. {\bf b.} Curve connecting the points A and B, $\overline{(A,B)}_{\M}$ (left plot) and its approximation $\overline{(A,B)}_{M}$ (right plot).}
\label{distance}
\end{figure}
This algorithm was originally designed as a routing tool for computing the shortest path connecting two nodes in a non-oriented graph with non-negative weights on the edges~\cite{dijkstra1959}. Here, $M$ is re-interpreted as a non-oriented graph $G_M$ in which the weights correspond to the lengths of the $ne$ edges connecting the $nv$ vertices. Specifically, the graph $G_M$ is a discrete function defined as $G_M: \{1,2,3,...,nv\}\times \{1,2,3,...,nv\} \to \R_+$ assuming the value of the euclidean distance for all couple of vertices connected by a single edge in $M$. The pseudo-code for computing the shortest path on $G_M$ between the node A and all other reachable nodes in $M$ follows: 
\begin{itemize}
\item[] {\bf function} dijkstra($G_M$,A)
\item[] $distA[A] = 0$, $distA[\cdot]$ is the array storing the distance from A
\item[] {\bf for} each vertex $v \in G_M$
\begin{itemize} 
\item[] $distA[v] \leftarrow \infty $
\item[] $prevA[v,1] \leftarrow $ null, $prevA[v,\cdot]$ is the array storing the nodes connecting A to v
\item[] {\bf if} $v \neq A$, add v to priority queue $Q_A$
\end{itemize}
\item[] {\bf while} $Q_A$ is not empty
\begin{itemize}
\item[] find $u: |x_A-x_u| = min_{j \in Q_A}|x_A-x_j|$, u is the vertex closer to A in $Q_A$
\item[] {\bf for} each unvisited neighbor v of u
\begin{itemize}
\item[] $tempDist \leftarrow distA[u] + edgeLength(u,v)$
\item[] {\bf if} $tempDist < distA[v]$
\begin{itemize} 
\item[] $nStpA[v] = nStpA[v]+1$, $nStpA[v]$ is a counter for the steps needed from A to v
\item[] $distA[v] = tempDist$
\item[] $prevA[v,nStpA[v]] = u$ 
\end{itemize}
\end{itemize}
\end{itemize}
\item[] {\bf return} distA[$1:nv$], prevA[$1:nv$,$nStpA(1:nv)$]
\end{itemize}
This algorithm returns the function $d_M(A,\cdot)$ corresponding to the polygonal approximation of $d_{\M}(A,\cdot)$ and is based on the hypothesis that any subpath $\overline{CD}$ of the unknown shortest path $\overline{AB}$ is also the shortest path between vertices C and D. This polygonal $\overline{(A,B)}_{M}$ is depicted in {\bf Figure.\ref{distance}b} and compared with the geodesic curve $\overline{(A,B)}_{\M}$. Dijkstra algorithm time complexity is O$(ne+nv\log(nv))$ while its space complexity is O$(nv)$~\cite{russell2002,bauer2010}. However, it is nowadays still one the most performing routing procedure. 

\subsection{Time discretization}
The full discretization of the semi-discrete problem \eqref{semiD} is obtained by adopting an implicit P(EC)$^k$ $\beta$-Newmark II order scheme. The solution at time level n+1 is obtained as follows for the i-th point of the mesh:

\begin{itemize}

\item P-Step (Prediction):
\begin{equation}
{\begin{cases}
\dot{\u}^P_{(i,n+1)} = \dot{\u}_{(i,n)} + (1-\gamma)\Delta t \ddot{\u}_{(i,n)},\\
\u^P_{(i,n+1)} = \u_{(i,n)} + \Delta t \dot{\u}_{(i,n)} + (1-2\beta)\frac{\Delta t^2}{2} \ddot{\u}_{(i,n)},\\
\dot{\u}^k_{(i,n+1)} = \dot{\u}^P_{(i,n+1)},\\
\u^k_{(i,n+1)} = \u^P_{(i,n+1)};
\end{cases}}
\end{equation}

\item (EC)$^k$-Step (implicit Evaluation-Correction): 

\begin{itemize}

\item E-Step: 
\begin{equation}
{\ddot{\u}^k_{(i,n+1)} = \frac{{\bf K}(\x^k_{(i,n+1)},\u^k_{(i,n+1)}) + {\bf b}(\x^k_{(i,n+1)},\u^k_{(i,n+1)})}{\rho_i}}\, ;
\end{equation}

\item C-Step:
\begin{equation}
{\begin{cases}
\dot{\u}^{k+1}_{(i,n+1)} = \dot{\u}^P_{(i,n+1)} + \gamma \Delta t \ddot{\u}^k_{(i,n+1)},\\
\u^{k+1}_{(i,n+1)} = \u^P_{(i,n+1)} + \beta \Delta t^2 \ddot{\u}^k_{(i,n+1)};
\end{cases}}
\end{equation}

\item Convergence Check:
\begin{equation}
{\begin{cases}
|| \dot{\u}^{k+1}_{(i,n+1)} - \dot{\u}^{k}_{(i,n+1)} ||_{L_2} \leq \epsilon, \\
\dot{\u}^k_{(i,n+1)} = \dot{\u}^{k+1}_{(i,n+1)},\\
\u^k_{(i,n+1)} = \u^{k+1}_{(i,n+1)}.
\end{cases}}
\label{convCheck}
\end{equation}
\end{itemize} 
\end{itemize}
The tolerance $\epsilon$ considered in this paper is equal to 10$^{-7}$, and the method converges in 2--8 iterations. {So that, if the error is below the tolerance, we set: \begin{itemize}
\centering
\item $\u_{(i,n+1)} = \u^{k+1}_{(i,n+1)}$;
\item $\dot{\u}_{(i,n+1)} = \dot{\u}^{k+1}_{(i,n+1)}$;
\item $\ddot{\u}_{(i,n+1)} = \ddot{\u}^{k+1}_{(i,n+1)}$.
\end{itemize}}

\section{\bf Numerical experiments}
\label{exp}

In this section, two numerical experiments are performed to detect the role played in the evolution problem by the two main ingredients affecting the present mathematical model: nonlocality (ruled by the parameter $\alpha$) and nonlinearity (ruled by the parameter $p$). The effects exhibited by the solutions of the Cauchy problem \eqref{eq:CP} are considered and critically analyzed for a closed two-dimensional manifold. The considered system is a spherical surface with unitary radius discretized \textit{via} a conforming mesh composed by $nt=642$ triangles, $nv=1280$ vertices and $ne=1920$ edges. The time step $\Delta t$ is fixed for all computations to $10^{-3}$, as well as for simplicity, the density assigned to each vertex is $\rho_i=1; i\in\{1, 2, 3, ..., nv \}$ and it is assumed constant in time. In the same fashion, the pairwise elastic modulus k and the peridynamic horizon $\delta$ are assumed independent from specific edges and vertices, namely $k_{i,j}=1; i, j \in\{1, 2, 3, ..., nv \}$ and $\delta = 0.5$. With these choices for the constitutive parameters, abilities and limitations of the proposed model are investigated through two different numerical experiments. Specifically, the evolution of the system as a function of $p$ and $\alpha$ is firstly studied when considering non-trivial initial conditions for the velocity field without external force, ${\bf b}({\bf x},t)=0$; then, the system initially at the rest is exposed to a uniaxial load along z. Abilities of the proposed model are measured by tracking in time the surface stretching, $\delta S(t) = \frac{S(t)-S(0)}{S(0)}$ - being $S(\cdot)$ the 2D surface measure; and the internal energy, $E(t)$, given by \eqref{eq:energy-alpha}. Indeed, the latter is considered as the sum of two terms, the  kinetic energy and the  potential energy, namely, 
\begin{equation}
\label{EkinEpot}
\begin{cases}\displaystyle E_{kin} (t) = \frac{\rho}{2}\int_{\M} |\pt \u(\x,t)|^2d\x \, ,\\
\displaystyle E_{pot}(t) = \frac{\kappa}{2p} \int_{\M}\int_{\M\cap B_\delta(\x)}
\frac{|\u(\x,t)-\u(\x',t)|^p}{d_{\M}(\x,\x')^{2+\alpha p}}\,d\mathcal{H}^2(\x')\,d\mathcal{H}^2(\x) \, . \end{cases}
\end{equation}
   
\subsection{Spherical surface undergoing  random initial velocity conditions}

The first numerical experiments campaign concerns the evolution of the described spherical surface when considering for each vertex a randomly chosen velocity vector within the ball by radius $|\pt\u(0)|$ (= 0.1, 0.5, 1.0) as a function of p (=2,3, and 5) and $\alpha$ (= 0.0001, 0.5, and 0.999). Note that, the proposed model is defined for $\alpha \in (0,1)$, so that, we use $\alpha =$ 0.0001 and 0.999 as extrema of $\alpha$ range of variation. 
\begin{figure}
\centering
\includegraphics[scale=0.21]{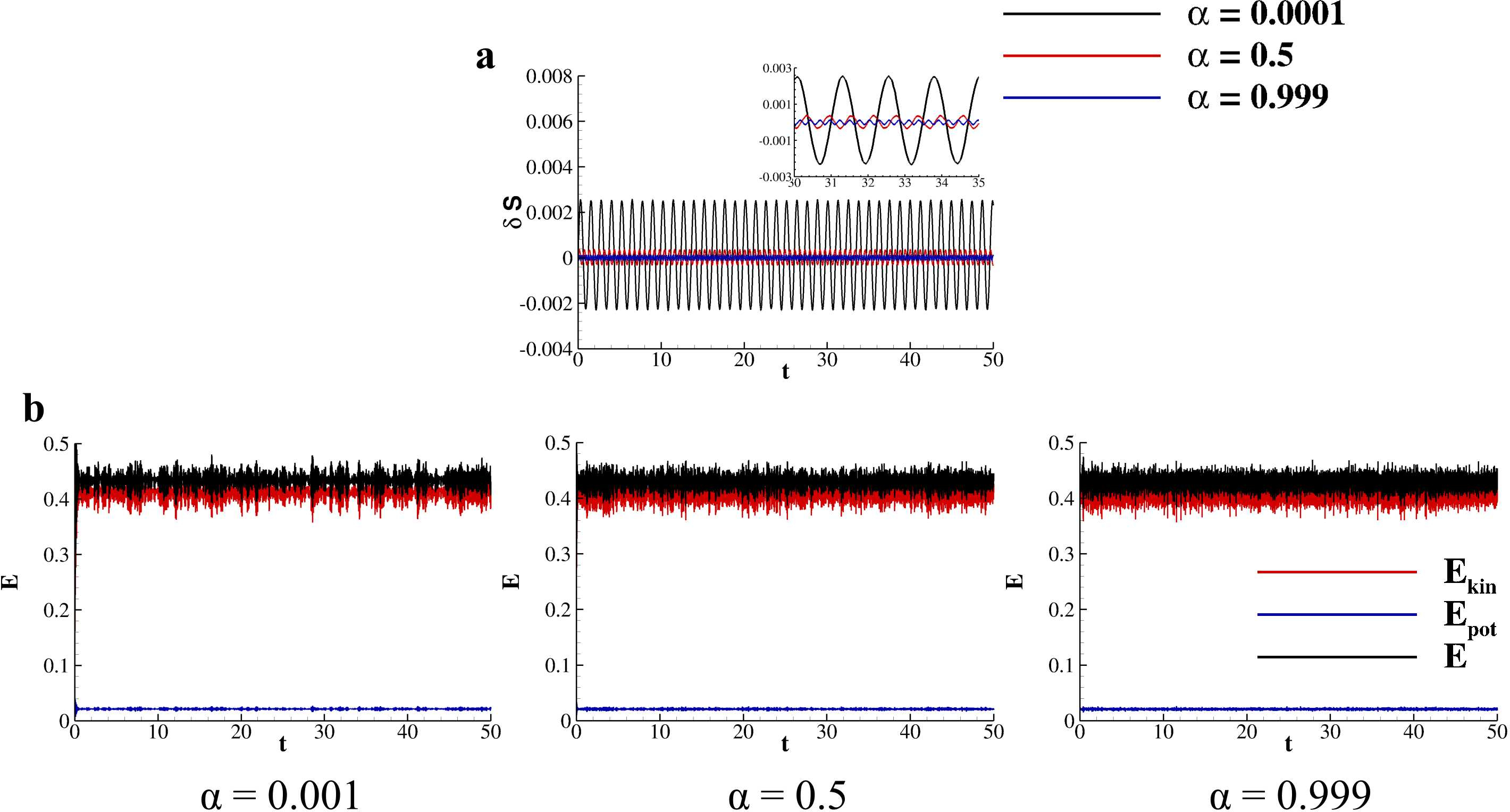}
\caption{{\bf Evolution of a spherical surface due to randomly distributed initial velocity conditions for $p=2$ and $|\pt\u(0)|=0.1$.} Distribution in time of the surface strain ({\bf a}) and energy ({\bf b}) for $\alpha=$ 0.0001, 0.5, and 0.999.}
\label{df01_random}
\end{figure}
\begin{figure}
\centering
\includegraphics[scale=0.21]{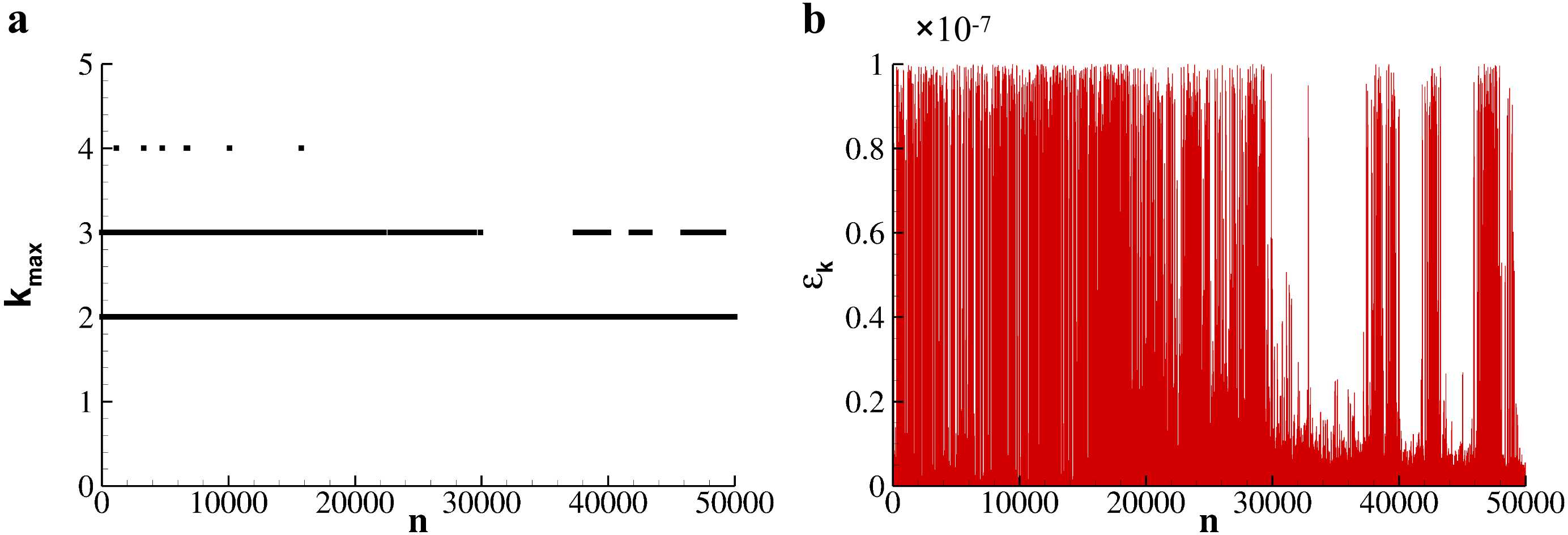}
\caption{{\bf Convergence history of the proposed implicit $P-(EC)^k$ scheme for $p=2$, $\alpha=0.0001$, and $|\pt\u(0)|=0.1$.} {\bf a.} Number of k-iterations performed for each time level n for a target convergence of $10^{-7}$. {\bf b.} Acquired convergence $\epsilon_k$ for each time level n. See \eqref{convCheck} for details.}
\label{convergence}
\end{figure}
For $|\pt \u(0)|=0.1$ and $p=2$ the system behave as a classical elastic material. This is not surprising since for $|\pt \u(0)|=0.1$ the initial amount of kinetic energy is quite small and for $p=2$ the nonlocal wave equation reduces to a linear equation. Due to initial conditions for $\pt \u ({\bf x})$, the spherical manifold deforms while responding to the constitutive equation of the material that restrains all displacement from the resting configuration. This results in a periodic pattern for the relative surface stretching $\delta S$ - reported for $\alpha=$ 0.0001, 0.5, and 0.999 in {\bf Figure.\ref{df01_random}a}. The period of such oscillation is 1.26, 0.61, and 0.265 for $\alpha =$ 0.0001, 0.5, and 0.999, respectively; while, indeed, the offset is the zero-stretching axis. Interestingly, for the linear case $p=2$, $\alpha$ behaves as a parameter for the fine-tuning of the material mechanical stiffness, and no clue of its fundamental role in enhancing nonlocality effects arise. Moreover, the distributions of the two-component of the internal energy (see \eqref{EkinEpot}) are documented in {\bf Figure.\ref{df01_random}b} proving that the proposed $P-(EC)^k$ II order scheme converges to a dissipative solution as defined in Sec.\ref{subsec:spaces}. While, the convergence history for $\alpha = 0.0001$ is documented in {\bf Figure.\ref{convergence}}, thus demonstrating that the method converges to the target tolerance $\epsilon=10^{-7}$ within 2-4 inner iterations.
\begin{figure}
\centering
\includegraphics[scale=0.225]{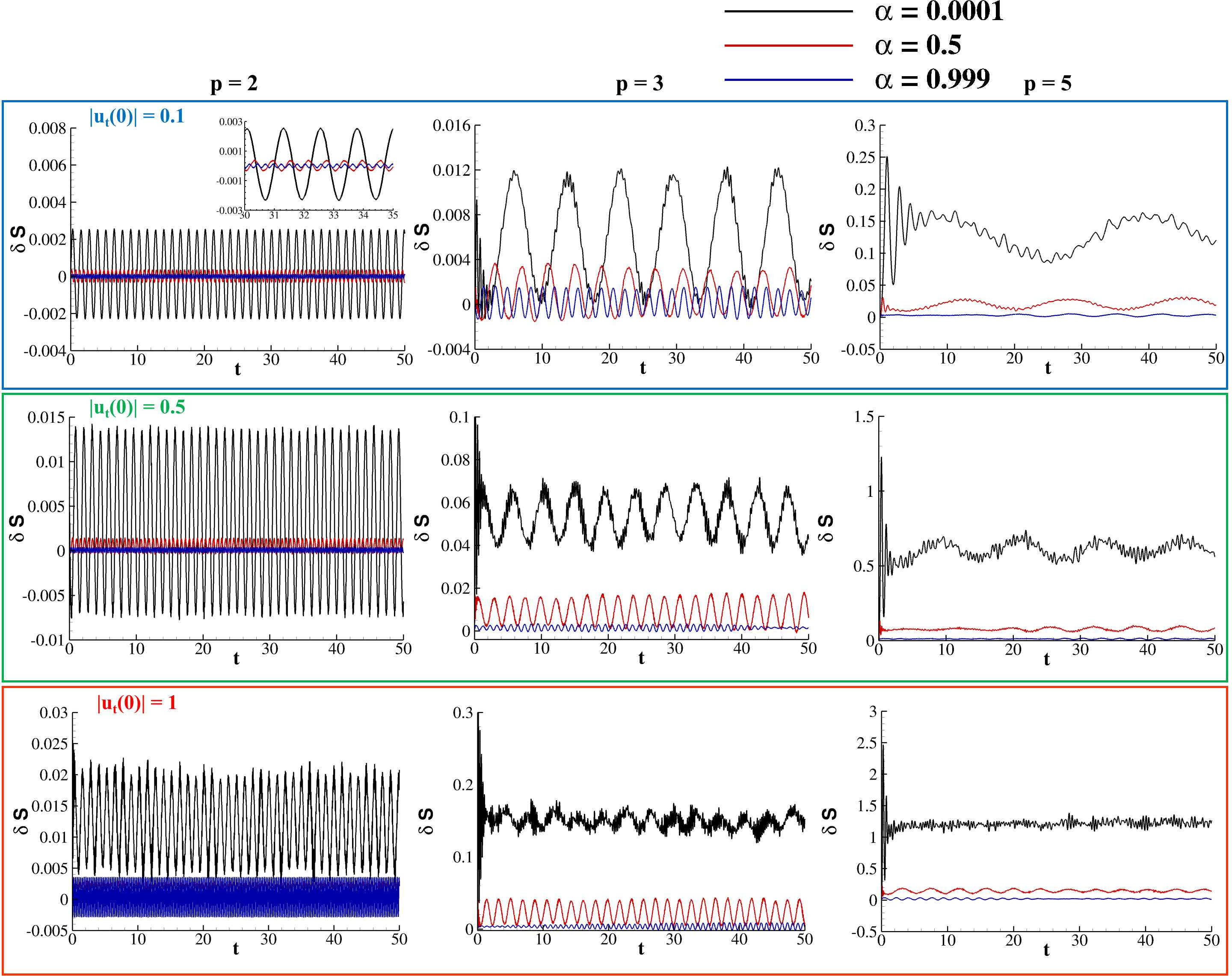}
\caption{{\bf Distribution of $\delta S(t)$ as a function of $p$, $\alpha$ and $|\pt\u(0)|$.}}
\label{allP_random}
\end{figure}
The role of $\alpha$ and $p$ is further deciphered by analyzing the distribution of $\delta S$ over time when systematically varying $p$, $\alpha$, and $|\pt\u(0)|$ (see {\bf Figure.\ref{allP_random}}). {Three scenarios appear: 
\begin{itemize}
\item for $p=2$ the material behave as a linear elastic solid and the distribution of $\delta S$ correspond to a periodic oscillation around a certain offset that depends on the effective mechanical stiffness of the material with respect to $|\pt \u(0)|$;
\item dispersive effects arises and the distribution of $\delta S$ appear to be a superposition of different signals (see plots corresponding to (p=3, $\alpha>0$) for $|\pt \u(0)|$=0.1 and 0.5 and plots corresponding to (p=3, $\alpha>0.5$) for $|\pt \u(0)|$=1.0);
\item the system is immediately dramatically damaged and the distribution of $\delta S$ appears as a dense band, a quasi-constant function with some superposed perturbations (see plots corresponding to (p=3, $\alpha=0.0001$) and $|\pt \u(0)|$=1.0, (p=5)).
\end{itemize}}
Those three scenarios, indeed, depend on both, the module of the imposed initial velocity, $|\pt\u(0)|$ and the characteristics of the material itself, $(p,\alpha)$. Specifically, by increasing $p$ the exponent of the numerator in the integral kernel  $({\bf K}\,\u)$ increases. This gives that for small relative displacements - $|\u_i-\u_j|$ - the response of the material is almost negligible. While, for large relative displacements the response of the material becomes intense. For fixed $p$, the power of $d_{\M}({\bf x'},{\bf x})$ in \eqref{eq:operator} is determined by $\alpha$ and, in turn, it gives the weight of nonlocal interactions. As a result of this complex picture, the enunciated three scenarios are produced.
\begin{figure}
\centering
\includegraphics[scale=0.21]{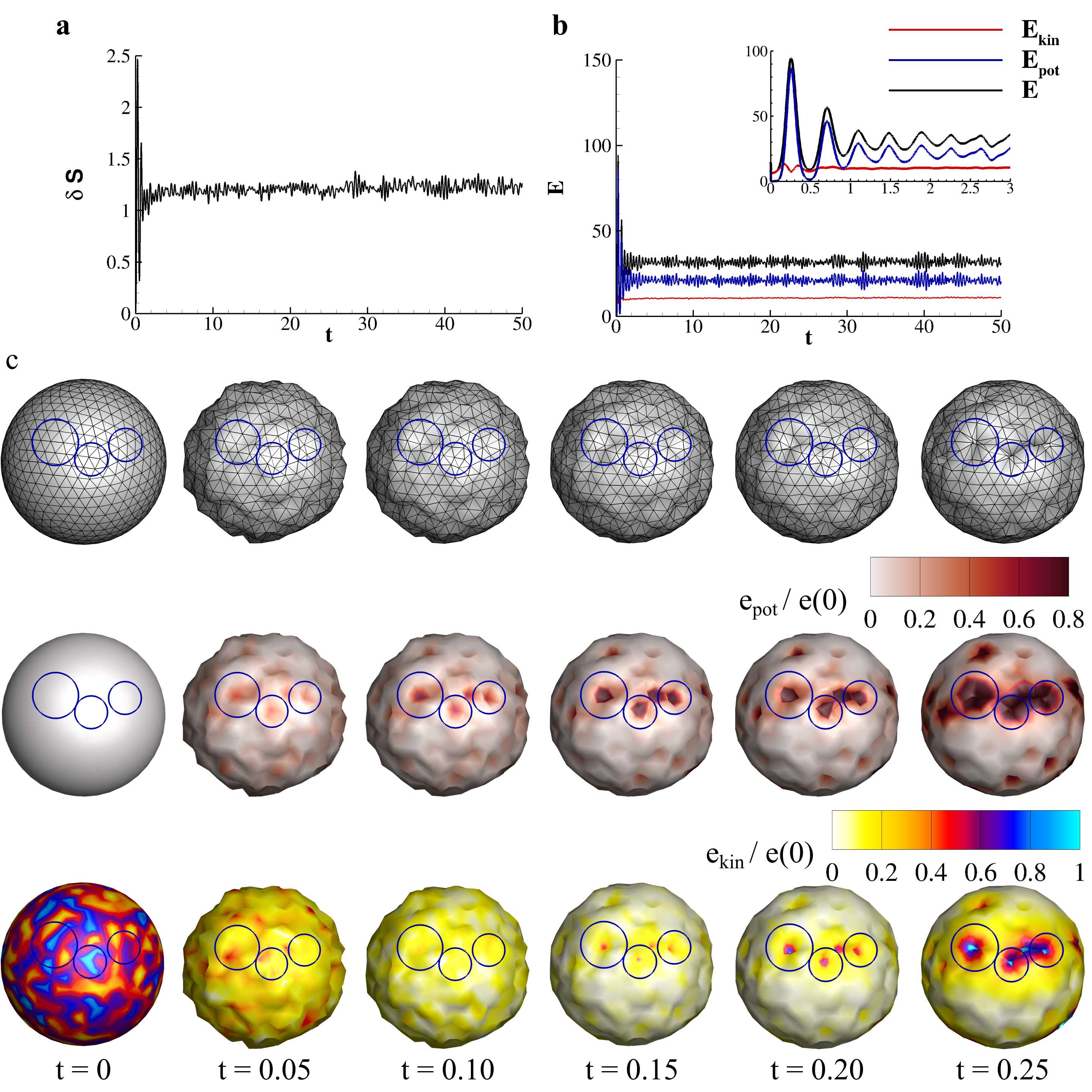}
\caption{{\bf Fractures generation in the spherical surface for $p=5$, $\alpha=0.0001$, and $|\pt\u(0)|=1$.} Distribution in time of the surface strain ({\bf a}) and energy ({\bf b}). ({\bf c}) System configuration taken at different time points providing the computational mesh deformation and contour plots of both kinetic and potential internal energy densities.}
\label{df1_random}
\end{figure}
The generation of {\it fractures} localized on the spherical surface corresponds to the formation of singularities in $\u$ that may be observed strictly for $p>2$ when challenging the system with a higher module of $|\pt \u(0)|$. As the matter of fact, for $p=2$ a singularity in $\u$ cannot be generated starting from a non-singular initial condition. Specifically, for $p=5$ and $|\pt \u(0)|=1$ - worst case scenario within the parameter space chosen - an harsh response of the system is triggered and a strong increase in $\delta S$ is observed for $t<0.3$ ({\bf Figure.\ref{df1_random}a}) as a consequence of the potential energy violent growth ({\bf Figure.\ref{df1_random}b}). Surface fractures appear in the regions in which the density of the internal potential energy $e_{pot}$ localizes - as shown in {\bf Figure.\ref{df1_random}c}. Specifically, for $t<0.15$  both energy densities results quite uniformly distributed on the manifold. On the contrary, for $t \geq 0.15$ the contour plots in {\bf Figure.\ref{df1_random}c} clearly show a net concentration of $e_{pot}$ in the enlighten blue circles delimiting three fracture zones. After this initial rush, the damaged system continues its evolution randomly corrugating the non-fractured surface, and the distribution of $\delta S$ for $t>1$ flattens. Interestingly, this harsh response documented for $E_{pot}$ and $\delta S$ is not found in $E_{kin}$. The kinetic energy density localizes in the fractured regions as well its potential counterpart, however, its integral remains controlled by the initial conditions (see {\bf Figure.\ref{df1_random}b}). Lastly, the authors verified that the positions in which such localizations appear due to the randomly chosen initial conditions do not follow any specificity. 

\subsection{Spherical capsule deforming under uniaxial load}
\begin{figure}
\centering
\includegraphics[scale=0.21]{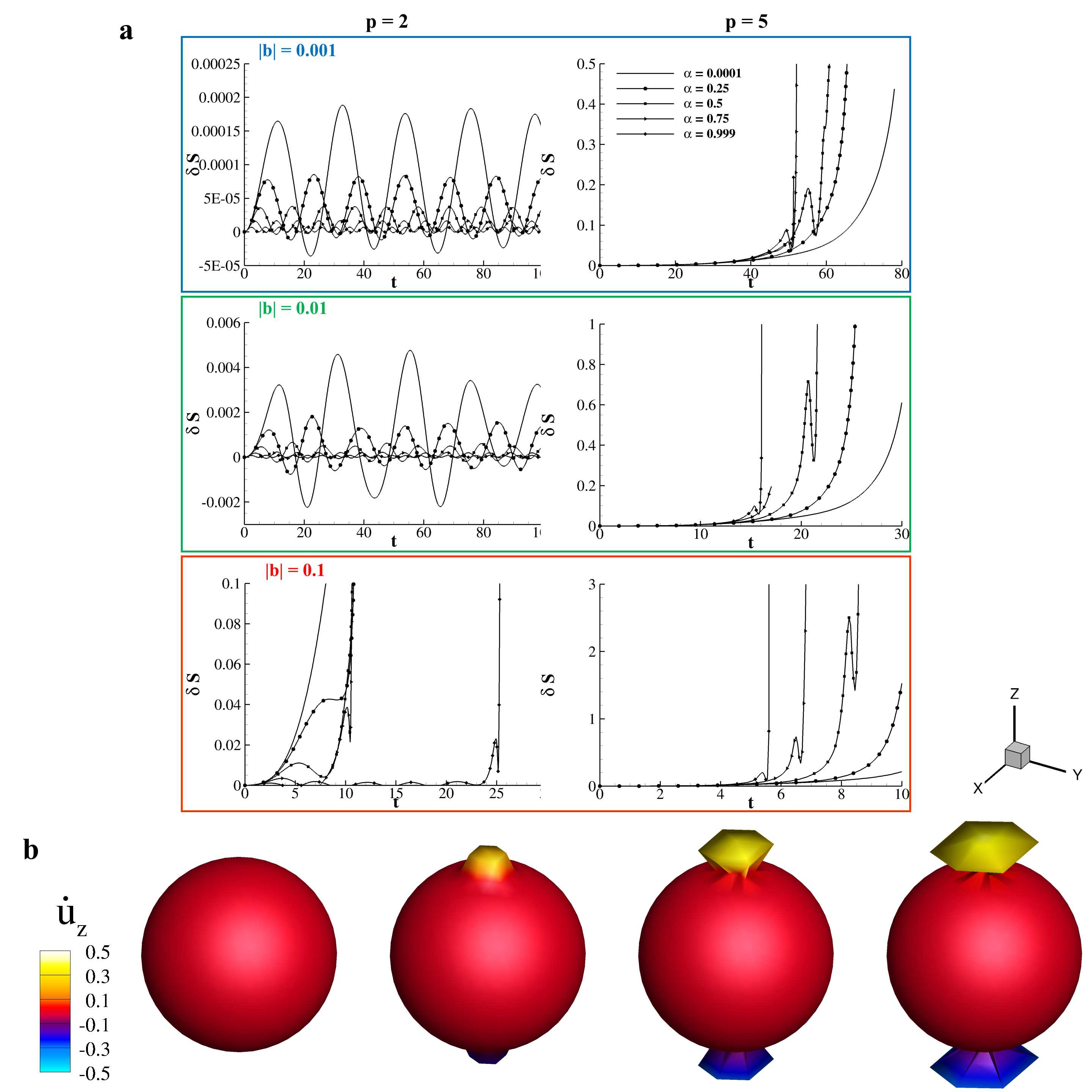}
\caption{{\bf Spherical capsule deforming under uniaxial load.} {\bf a.} Distribution of $\delta S$ as a function of $\alpha$ and $|{\bf b}|$ for $p=$ 2 and 5. {\bf b.} Configurations obtained for $\alpha = 0.5$ and $|{\bf b}|=0.1$ taken at $t$ = 0, 5, 8.75, and 9 with superposed contour levels of $\dot{\u}_z$.}
\label{uniload}
\end{figure}
The system is analyzed also in presence of external forcing. Specifically, the evolution of the spherical manifold deforming due to a constant load parallel to z, ${\bf b}$, is computed. The load is applied to all vertices having coordinates (0,0,1) and (0,0,-1), with a tolerance of 0.05, being 1 the spherical capsule radius. Firstly, the behavior of the system for the linear case, $p=2$, is assessed as a function of $\alpha$ and $|{\bf b}|$, as documented in {\bf Figure.\ref{uniload}a}. For $|{\bf b}| \leq 0.01$ the material response  and the functional form of $\delta S$ correspond to an oscillation whose characteristics depend on $\alpha$. However, in this specific case some modulations in such oscillations are observed. Moreover, the role of $\alpha$ as a parameter for tuning the material stiffness clearly emerges. In fact, as $\alpha$ approaches to 1 a strong regularization of $\delta S$ oscillations is observed and dispersive effects seems to become negligible. All peaks in $\delta S$ distributions obtained for $|{\bf b}|=0.001$ assess to 3.75$\times 10^{-5}$, 1.65$\times 10^{-5}$ and 0.8$\times 10^{-5}$, for $\alpha = 0.5$, 0.75 and 0.999 respectively. On the contrary, a peak value of 8.27$\times10^{-5}$ with a modulation of $\pm 7.5\times10^{-5}$ is obtained for $\alpha=0.25$ and $1.8\times10^{-4}\pm 0.8\times10^{-4}$ for $\alpha = 0.0001$. Analogous trends are observed for $|{\bf b}|=0.01$ while for $|{\bf b}|=0.1$ the system damages for all $\alpha$ as demonstrated in {\bf Figure.\ref{uniload}a}. As discussed for the previous numerical experiment, by increasing $p=5$ the material   response to stimuli becomes stiffer and fractures are generated easily regardless of $\alpha$. Specifically, the exponent of the numerator in the kernel of ${\bf K}(\u)$ increases linearly with $p$. So that, for small displacements the response of the material is negligible while for large displacements the response becomes harsh. With the same picture, here for $p=5$,  the system exhibits fracture regardless from $\alpha$ and $|{\bf b}|$ as demonstrated in {\bf Figure.\ref{uniload}b}. For completeness the configurations  acquired during the evolution are reported for $|{\bf b}|=0.1$ and $\alpha=0.5$ in {\bf Figure.\ref{uniload}c}. 
           
\section*{\bf Conclusions}

The analysis of the Cauchy problem for the peridynamics evolution of a two-dimensional closed smooth manifold has been deepened in this work through a suitable discretization leading to a detailed numerical analysis. The consistency of such discretization has been rigorously proved and various numerical experiments have contributed to clarifying some questions characterizing the creation of singularities. In particular, the role played by the non-local character of the internal interaction has been highlighted by showing that strong non-locality (quantified by small values of $\alpha$) leads inexorably to the appearance of severe growth and blow-up, regardless of the strength of the non-linearity (quantified by high values of $p$) of mechanical relevant quantities (surface strain and energy). On the other hand, strong nonlinearities induce high energy concentration whose space location, at this stage, seems to be unpredictable and deserves further studies. 
 
\section*{\bf Acknowledgments}

AC and TP are members of Gruppo Nazionale per il Calcolo Scientifico (GNCS), GMC and FM of Gruppo Nazionale per l'Analisi Matematica, la Probabilit\`{a} e le loro Applicazioni (GNAMPA) of the Istituto Nazionale di Alta Matematica (INdAM). GMC expresses its gratitude to HIAS - Hamburg Institute for Advanced Study for their warm hospitality.

This work was partially supported by:
\begin{itemize}
\item the project ``Research for Innovation'' (REFIN) - POR Puglia FESR FSE 2014-2020 - Asse X - Azione 10.4 (Grant No. CUP - D94120001410008);
\item INdAM – GNCS ``Project Metodi numerici per modelli descritti mediante operatori differenziali e integrali non locali'' Grant No. CUP E53C22001930001;
\item Italian Ministry of Education, University and Research under the Programme Department of Excellence Legge 232/2016 (Grant No. CUP - D94I18000260001).
\end{itemize}    


\end{document}